\documentclass[a4paper,11pt,reqno]{amsart}

\usepackage{amssymb,amscd,amsthm,mathrsfs,comment}
\usepackage[all]{xy}         
\usepackage[centertags]{amsmath}
\usepackage{breqn}
\usepackage{csquotes}

\usepackage{graphicx}

\newtheorem{theo}{Theorem}[section]

\newtheorem{coro}[theo]{Corollary}
\newtheorem{prop}[theo]{Proposition}
\newtheorem{lemm}[theo]{Lemma}

\newtheorem{claim}[theo]{Claim}
\newtheorem{teo}[theo]{Theorem}
\newtheorem{add}[theo]{Addendum}
\newtheorem{sch}[theo]{Scholium}

\newtheorem{quest}[theo]{Question}
\newtheorem{problem}[theo]{Problem}

\theoremstyle{definition}
\newtheorem{rema}[theo]{Remark}
\newtheorem{defi}[theo]{Definition}
\numberwithin{equation}{section}

   \def\DD{{\mathbb D}}

 \def\RR{{\mathbb R}}  \def\TT{{\mathbb T}}
   
 \def\ZZ{{\mathbb Z}}

  \def\cG{\mathcal{G}} \def\cM{\mathcal{M}} 
  \def\cH{\mathcal{H}}  \def\cT{\mathcal{T}}
\def\cC{\mathcal{C}}    \def\cU{\mathcal{U}}
   \def\cP{\mathcal{P}} 
    
\def\cF{\mathcal{F}}

\newcommand{\Ecs}{E^{cs}}
\newcommand{\Euu}{E^{uu}}
\newcommand{\Ess}{E^{ss}}
\newcommand{\Ecu}{E^{cu}}

\newcommand{\inv}{^{-1}}
\newcommand{\subof}{\subset}
\newcommand{\lam}{\lambda}
\newcommand{\gam}{\gamma}
\newcommand{\Coneuu}{\cC^{uu}}
\newcommand{\Conecs}{\cC^{cs}}
\newcommand{\qandq}{\quad \text{and} \quad}

\newcommand{\eps}{\varepsilon}

\newcommand{\ie}{{\it i.e., } }
\newcommand{\eg}{{\it e.g., } }

\renewcommand{\xrightarrow}[1][1]{\underset{#1}\rightsquigarrow}  

\author[C. Bonatti]{Christian Bonatti} 
\address{Institut de Math. de Bourgogne CNRS - URM 5584, Universit\'e de Bourgogne Dijon 21004, France}

\author[A. Gogolev]{Andrey Gogolev} 
\address{The Ohio State University, Columbus, OH 43210, USA}

\author[A. Hammerlindl]{Andy Hammerlindl}
\address{ School of Mathematics, Monash University, Victoria, Australia}

\author[R. Potrie]{Rafael Potrie}
\address{CMAT, Facultad de Ciencias, Universidad de la Rep\'ublica, Uruguay}
\curraddr{Institute for Advanced Study, Princeton, NJ 08540, USA}

\title[Anomalous partially hyperbolic examples III]{Anomalous partially hyperbolic diffeomorphisms III: abundance and incoherence}
\thanks{C.B. was partially supported by IFUM. A.G. was partially supported by NSF grant DMS-1266282 and Simons grant 427063; R.P. was partially supported by CSIC group 618.
This research was partially supported by the Australian Research Council}

\begin{document}
\maketitle

\begin{abstract}
Let $M$ be a closed 3-manifold which admits an Anosov flow. In this paper we develop a technique for constructing partially hyperbolic representatives in many mapping classes of $M$.
We apply this technique both in the setting of geodesic flows on closed hyperbolic surfaces and for Anosov flows which admit transverse tori.
We emphasize the similarity of both constructions through the concept of \emph{$h$-transversality}, a tool which allows us to compose different mapping classes while retaining
partial hyperbolicity.

In the case of the geodesic flow of a closed  hyperbolic surface $S$ we build stably ergodic, partially hyperbolic diffeomorphisms whose mapping classes form  a subgroup of
the mapping class group $\cM(T^1S)$ which is isomorphic to  $\cM(S)$. At the same time we show that the totality of mapping classes which can be realized by partially hyperbolic diffeomorphisms does not form a subgroup of $\cM(T^1S)$.  

Finally, some of the examples on $T^1S$ are absolutely partially hyperbolic, stably ergodic and robustly non-dynamically coherent, disproving a conjecture in 
\cite{HHU2}.

{ \medskip \noindent \textbf{Keywords:} partially hyperbolic
diffeomorphisms, 3-manifolds, classification, stable
ergodicity, dynamical coherence. 

 }
\end{abstract}

\section{Introduction}

A diffeomorphism $f\colon M \to M$ of a closed Riemannian manifold is \emph{partially hyperbolic} 
if the tangent bundle $TM$ splits into three (non-trivial) $Df$-invariant continuous subbundles 
$TM= E^{ss}\oplus E^c \oplus E^{uu}$ such that for some $\ell>0$, for all $x \in M$ and all unit vectors $v^\sigma \in E^{\sigma}(x)$ ($\sigma=ss,c,uu$) one has 
$$ \|Df^\ell v^{ss} \| < \min \{1, \|Df^\ell v^c\|\} \leq
\max \{1, \|Df^\ell v^c\| \} < \|Df^\ell v^{uu}\|.
$$

Sometimes, one uses the stronger notion of \emph{absolute partial hyperbolicity}. 
This means that $f$ is partially hyperbolic and there exist constants $\lambda_1 < 1 < \lambda_2$ such that

$$ \|Df^\ell v^{ss} \| < \lambda_1<  \|Df^\ell v^c\| < \lambda_2 < \|Df^\ell v^{uu} \|. $$
\medskip
The subbundles $E^{ss}$, $E^c$ and $E^{uu}$ depend on $f$ and we will indicate this, when needed, by using a subscript, \eg $E^{ss}_f$.

Partially hyperbolic diffeomorphisms appear naturally in many contexts and provide natural classes to study phenomena 
such as robust transitivity and stable ergodicity (see \eg \cite{BDV,Wilkinson,CHHU,HP-survey}). For a long time, 
the known examples in dimension three were rather restrictive, and efforts were made to try to show that the behavior seen in that restricted family of 
examples was indeed general: such results hold, for instance, on manifolds with solvable fundamental group (see \cite{HP}). 
Recently new examples have started to appear (\cite{HHU2,BPP,BGP}).

In this paper, we consider a closed $3$-manifold and look for mapping classes  
(\ie the diffeomorphisms up to homotopy) which contain 
partially hyperbolic representatives. In fact, the constructions we present here all start with an Anosov flow, 
and therefore we consider $3$-manifolds which support Anosov flows. For some of the new examples we study specific properties of 
the dynamics and geometry that makes them different from the previously known behaviors (see Theorem~\ref{t.incoh} below).  

A general fact about a $3$-manifold $M$ is that its mapping class group $\cM(M)\colon=\pi_0(\mathrm{Diff}^1_+(M))$ is strongly 
related to the group generated by the Dehn twists associated
to incompressible tori (see~\cite{Johannson}). We will consider two cases which correspond to different positions of the tori
with respect  to the Anosov flow. 

In Theorem~\ref{t.transverse}, the tori will be transverse to the flow. 
In Theorem~\ref{t.geodesic} the Anosov flow will be a geodesic flow and we will perform Dehn twist along tori over closed geodesics: 
these tori are not transverse to the geodesic flow, but they are \emph{Birkhoff sections} for the flow.

These two cases already appear in \cite{BGP}.
The novelty here is that we are able to compose an arbitrary number of such 
Dehn twists, and that this involves a conceptual understanding of the mechanism underlying the examples in \cite{BPP,BGP}. Note that a priori the mapping classes which admit partially hyperbolic representatives do not form a subgroup of $\cM(M)$ (\eg when $M= \TT^3$, it is not a subgroup). In both settings described above, this paper exhibits infinite 
subgroups of $\cM(M)$ such that each element admits a partially hyperbolic representative.


In recent years, much of the study of partially hyperbolic systems 
has focused on questions of robust transitivity and stable ergodicity.
When analyzing such systems,
an important first property to establish
is whether or not there exists
an invariant foliation tangent to the center direction.
Indeed, a long-standing open question, recently answered by Rodriguez Hertz,
Rodriguez Hertz and Ures was whether a partially hyperbolic system
with one-dimensional center necessarily had a center foliation (\cite{HHU2}).
They showed this was not the case by constructing a counterexample on the
3-torus.
This surprising and important result introduced new techniques
of constructing partially hyperbolic example in the non-transitive setting.
Their construction fundamentally relies on having an embedded 2-torus
tangent to the center-stable direction along which the dynamical incoherence occurs.
As this torus is normally attracting,
it cannot be used in the construction of a stably ergodic
or robustly transitive example.
Further, they conjectured (\cite{HHU2,CHHU}) that for transitive, partially hyperbolic systems
in dimension three, invariant center foliation must always exist.

In fact, this conjecture is false as a consequence of some of the new examples we present in this paper. Our results imply the following:


 \begin{displayquote}
     \emph{ There is a $C^1$-open set of partially hyperbolic diffeomorphisms which are both transitive and dynamically incoherent.}
  \end{displayquote}

These examples can also be made conservative and stably ergodic.
See  Theorem \ref{t.incoh} for a precise statement.
These systems do not have center-stable or center-unstable foliations.
However, they do possess unique invariant branching foliations
as defined by Burago and Ivanov \cite{BI}.
Because of this,
we may discuss the branching center foliation
defined by intersecting the center-stable and center-unstable branching
folations. 

As a consequence of transitivity,
this ``branching'' occurs everywhere.
In particular,
for any non-empty open subset $U$ of $M,$
one can find distinct center leaves $L_1$ and $L_2$
such that $L_1 \cap L_2 \cap U$ is non-empty. 

Despite the branching,
these center leaves are comparable to the orbits of the Anosov geodesic flow 
defined on $M;$
there is a bijective correspondence between the leaves of the branching
center foliation and the orbits of the flow.
In particular, compact center leaves are associated to periodic orbits of the
flow.
However, the partially hyperbolic dynamics on these leaves
behaves very differently from the flow.

\begin{displayquote}
\emph{There is a $C^1$open set of partially hyperbolic diffeomorphisms such that the union of circles tangent to $E^c$ is dense in $M$, but none of these circles is periodic. }
\end{displayquote}

In light of this result,
any center leaf which has a periodic point must be a  line.
In previously-known examples,
the dynamics on such periodic center lines is either
contracting on large scales or expanding on large scale 
as depicted in Figure~\ref{fig:arrows} (a) and (b).
The new systems studied in the current paper
have periodic center lines which are contracting on one end
and expanding on the other, as shown in Figure \ref{fig:arrows}~(c).
This large-scale behaviour of the leaves is a key tool in proving
dynamical incoherence. 

\begin{figure}[t]
\vspace{0.5cm}
\begin{center}
\includegraphics[scale=1.0]{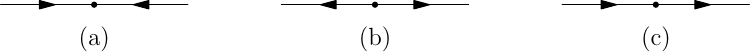}
\end{center}
\vspace{0.2cm}
\caption{\small{Possible dynamics on invariant center curves.
    Previous examples, such as on the 3-torus,
    have large scale dynamics as in (a) and (b).
    Examples in the current paper also have dynamics as shown in (c). Notice that there may be more than one fixed point in the leaves, but they should all lie in a compact interval.}\label{fig:arrows}}
\end{figure}
We also point out, that while in ``small manifolds'' a stronger notion called \emph{absolute} partial hyperbolicity is enough to guarantee dynamical coherence (see \cite{BBI-coherence,HP}) our new examples are absolutely partially hyperbolic, thus producing the first such examples of dynamical incoherence in the three-dimensional setting.

We begin by statings our result in the geometric setting.  


\subsection{Partially hyperbolic diffeomorphisms on the unit tangent bundle of a hyperbolic surface}

Let $S$ denote the surface of genus $g\ge2$, and let $T^1S$ denote its unit tangent bundle (equivalently, the circle bundle over $S$ whose Euler class is $2-2g$,  
the Euler characteristic of $S$).  Let $\cM(S)$ denote the group of homotopy (or equivalently, isotopy) classes of orientation preserving diffeomorphisms of 
$S$; in other words, 
$\cM(S)=\pi_0(\mathrm{Diff}_+(S))$.  

Given a diffeomorphism $\varphi: S \to S$, its normalized differential $T\varphi(v) = \frac{D\varphi v}{\|D\varphi v\|}$ is a diffeomorphism of $T^1S$ and so $\mathrm{Diff}_+(S)$ acts on the unit tangent bundle $T^1S$, yielding an injective homomorphism 
$\iota\colon\cM(S)\hookrightarrow \cM(T^1S)$ (cf. \cite[Proposition 25.3]{Johannson}). 

The main theorem in this setting is the following.

\begin{theo}\label{t.geodesic} For any $\varphi\in \cM(S)$ there exists a diffeomorphism $f\colon T^1S\to T^1S$  such that:
\begin{itemize}
 \item $f$ belongs to $\iota(\varphi)$,
 \item $f$ is absolutely partially hyperbolic, 
 \item $f$ preserves the volume and is stably ergodic.
\end{itemize}
\end{theo}

In section \ref{ss.obstruction} we use the results from \cite{HaPS} to see that not every mapping class of $T^1S$ can be realised and moreover, that the set of mapping classes which are realizable do not form a subgroup of $\cM(T^1S)$. 

\subsection{Dynamical incoherence}

For the examples given by Theorem~\ref{t.geodesic} applied to a \emph{pseudo-Anosov} mapping class $\varphi$   (see \eg \cite{FaMa} for a definition of pseudo-Anosov homeomorphism)
we are able to obtain properties for $f$ which were unexpected. 
Recall that a partially hyperbolic diffeomorphism is called \emph{dynamically coherent} if there exist $f$-invariant foliations $\cF^{cs}$ and $\cF^{cu}$ 
tangent to $E^{cs}:=E^{ss}\oplus E^c$ and $E^{cu}:=E^c\oplus E^{uu}$ respectively (and therefore, there is also an invariant center foliation obtained by taking intersection).  Otherwise, we say that $f$ is {\it dynamically incoherent.}

Recall that the strong stable and unstable bundles are always uniquely integrable for dynamical reasons (see \eg \cite{HP-survey} and references therein). In higher dimensions, when the center bundle has dimension at least two, 
there are examples where the center bundle is smooth but does not satisfy the Frobenius integrability condition (see \cite{BW-coherence}). When the center bundle is one-dimensional
the problem of integrability only comes from the lack of regularity. Assuming absolute partial hyperbolicity, geometric conditions on the 
strong foliations (which are always satisfied on the torus $\TT^3$) are known to imply dynamical coherence (see \cite{Brin,BBI-coherence}).
For certain families of 3-manifolds, exact conditions for dynamial coherence
are known
(see \cite{CHHU,HP-survey} and references therein).
However, there is no known general criterion for deciding whether the center
bundle is integrable.

This motivated \cite{BI} to construct objects called \emph{branching
foliations}
which exist for any 3-dimensional partially hyperbolic system
and serve as substitutes for the true invariant foliations.
Originally, these branching foliations were used only as a tool to
establish dynamical coherence.
Later, \cite{HHU2} constructed a concrete example which has branching
foliations, but not true foliations.
For this example the branching occurs
exactly on a finite collection of attracting or repelling  $2$-tori (which is incompatible with transitivity and absolute partial hyperbolicity).
Further, \cite{HHU2} conjectured that the existence of such tori is the unique obstruction for dynamical coherence (see also~\cite{CHHU}). 

Theorem~\ref{t.incoh} provides counterexamples to this conjecture:

\begin{theo}\label{t.incoh} For any pseudo-Anosov mapping class $\varphi\in \cM(S)$ there exists a diffeomorphism $f\colon T^1S\to T^1S$ such that:
\begin{itemize}
 \item $f$ belongs to $\iota(\varphi)$,
 \item $f$ is absolutely partially hyperbolic, 
 \item $f$ preserves the volume and is stably ergodic,
 \item $f$ is robustly transitive, 
 \item $f$ is robustly dynamically incoherent.
\end{itemize}
\end{theo}

For these examples we also establish minimality of the strong stable and unstable foliations. This implies that the branching sets of the branching foliations are dense in
the whole manifold. Distributions with this kind of behavior were already known to exist but not associated to partially hyperbolic dynamics~\cite{BonattiFranks}. 

It is natural to expect
that given a maximal curve $\eta$ tangent to center distribution
$E^c$
the stable saturation $W^s(\eta)$ is a center-stable leaf.
However, it was realized early on~\cite{BW} that theoretically the stable saturation $W^s(\eta)$ may fail 
to coincide with the full center-stable leaf. This may happen if $W^s(\eta)$ is not \emph{complete}. The example in~\cite{HHU2} also showed that such hypothetical behavior indeed may occur: $W^s(\eta)$ are not complete and their non-empty boundary is contained in the repelling tori. One could believe that the minimality of the strong foliations could imply completeness, but our examples show that this is not so.

Another important tool for the study of the dynamics of a partially hyperbolic diffeomorphism is the understanding of the dynamics in the center leaves. 
Our examples exhibit a new type of behaviour:  one can show that there are many compact center curves but  none of them is periodic, and that there exists simultaneously contracting, 
repelling and saddle node center leaves (see \eg~\cite[Section 7.3.4]{BDV}). In particular, the examples give a complete answer to~\cite[Problem 7.26]{BDV} (see also~\cite[Section 3.8]{BGP}). See Remark~\ref{rem.centerleafdyn} for a detailed explanation of these phenomena.

\subsection{Partially hyperbolic diffeomorphisms associated to a family of tori transverse to an Anosov flow}

 Given a torus $T$ embedded in a $3$-manifold $M$, we associate an element $\tau(T,\gamma)\in\cM(M)$ called \emph{the Dehn twist 
along $T$ directed by $\gamma$} to each homotopy class of closed curves $\gamma\in\pi_1(T)$ (see section~\ref{ss-dehn} for a precise definition). 

Given a torus $(T,\cF,\cG)$ endowed with a pair of transverse $C^1$-foliations $\cF$ and $\cG$, we consider the subset $G(T,\cF,\cG)$ of 
homotopy classes $\gamma\in\pi_1(T)$ such that there exist a $C^1$-continuous loop $\{\psi_t\}_{t\in \RR/\ZZ}$ of diffeomorphisms of 
$T$ such that:
\begin{itemize}
 \item $\psi_0=id$
 \item $\psi_t(\cF)$ is transverse to $\cG$ for all $t\in\RR/\ZZ$
 \item the homotopy class of the loop $\{\psi_t(x)\}_{t\in \RR/\ZZ}$, $x\in T$, is $\gamma$.  
\end{itemize}
It was shown in \cite{BZ} that either $G(T,\cF,\cG)=\pi_1(T)\simeq \ZZ^2$ or 
both foliations contain circle leaves in the same homotopy class and $G(T,\cF,\cG)$ is the cyclic group $\ZZ$ generated by this homotopy class. 

Now, if $X$ is a $C^2$-Anosov flow on a $3$-manifold $M$ and if $T\subset M$ is a torus transverse to $X$ then the center-stable and unstable foliations
of $X$ induce a pair of transverse $C^1$-foliation $(F^s_T,F^u_T)$ on $T$. Then we denote by $G(T,X)$ the group $G(T,F^s_T,F^u_T)$. 

Given an Anosov flow $X$ on a $3$-manifold $M$ we denote by $\cT(X)\subset \cM(M)$ the subgroup generated by the Dehn 
twists $\tau(T,\gamma)$ where $T$ is a torus transverse to $X$ and $\gamma$ belongs to $G(T,X)$.

\begin{theo}\label{t.transverse}
Let $X$ be an Anosov flow on a closed $3$-manifold $M$.  Then every element of $\cT(X)$ contains an absolutely
partially hyperbolic representative. 
\end{theo}

In certain cases, this theorem produces partially hyperbolic representatives in virtually all mapping classes.

\begin{coro}
\label{corollary}
For each $n\ge 1$ there exists a closed 3-dimensional graph manifold $M_n$ and a finite index abelian subgroup $G\simeq \ZZ^{2n}\oplus \ZZ/n\ZZ$ of the mapping class group $\cM(M_n)$ such that each mapping class in $G$ can be represented by an absolutely partially hyperbolic, volume preserving diffeomorphism.
\end{coro} 

Manifolds $M_n$ above are $n$-fold cyclic covers of the ambient manifold of the Bonatti-Langevin Anosov flow~\cite{BL}.

\subsection{Questions}
It is our opinion that this paper has opened even more questions than the previous one~\cite{BGP}. In particular, it remains a challenge to better understand these new examples which have several new features which contrast our previous beliefs and should be now taken into account in the study of the dynamics of partially hyperbolic diffeomorphisms. 

We state here some questions which we believe arise naturally from our work, but the reader is of course invited to add more. 

We have constructed many mapping classes which admit partially hyperbolic representatives on manifolds which admit Anosov flows. It is natural to ask for characterization of all such mapping classes.

\begin{problem}\label{prob1} Given a closed 3-manifold $M$ admitting an Anosov flow, determine which mapping classes admit partially hyperbolic representatives.\end{problem}

In this direction, the following question is natural:

\begin{quest}
%
    For a manifold $M$, let $\cP$ be those elements of $\cM(M)$
    admitting a partially hyperbolic representative.
    When is $\cP$ a subgroup of $\cM(M)$?
    When is $\cP$ a finite-index group?
    When does it equal $\cM(M)$?
\end{quest} 

In this direction, Corollary \ref{corollary} shows that on some 3-manifolds, the classes which can be realized by partially hyperbolic diffeomorphisms contain a finite index subgroup. See also subsection \ref{ss.obstruction} for results in this direction.

This paper reduces the realization of mapping classes admitting partially hyperbolic diffeomorphisms to understanding geometry of the subbundles and how can they be made $h$-transverse (see Definition~\ref{defi-htran}). Under certain assumptions on the Anosov flow, we have tools to check that certain Dehn twists along certain tori respect transversality between the subbundles of the Anosov flow. In particular, our understanding seems to be complete in the case when tori are transverse to the Anosov flow. However, there are other configurations where $h$-transversality is not completely understood. In particular, we address the case of Birkhoff tori only in the geometric setting --- we perform Dehn-twists using the geometric considerations of~\cite{BGP}.  The general case of Birkhoff tori remains to be understood.

\begin{problem} Determine which Dehn twists 
respect transversalities between the bundles of an Anosov flow. 
\end{problem}

Understanding the above problem may yield even more mapping classes with partially hyperbolic representatives and provide a way to attack Problem~\ref{prob1}. Notice that there are some general results about Anosov flows which indicate that the study of Birkhoff tori might be sufficient (see~\cite{Barbot}). 

Further questions arise about the geometry and dynamics of these new examples.
Recall that for certain examples which realize pseudo-Anosov mapping class we are able to establish dynamical incoherence as well as non-completeness of stable manifolds of center leaves (see Corollary~\ref{c.noncomplete}). Also these examples possess saddle-node center leaves. However, for the general examples, many questions arise:

\begin{quest}
    Which of the new examples are dynamically coherent (compare \cite{BPP})?
    Are the dynamically coherent examples plaque-expansive?
    Can the transitive examples transverse to tori be always made volume preserving? Robustly transitive?
\end{quest} 

The examples given by Theorem~\ref{t.incoh} come from the pseudo-Anosov mapping class and, hence, have some strong and better understood dynamical and geometric properties. A natural question that one can pose is:

\begin{quest}
    Is every absolutely partially hyperbolic diffeomorphism of $T^1S$ homotopic to the lift of a pseudo-Anosov diffeomorphism of $S$ transitive (or at least chain-recurrent)? 
\end{quest}    

Compare with \cite{BG,Shi,Pot}. 

\subsection{Outline of the paper} 

The paper is organized as follows. In section~\ref{s.criterion} we present a flexible criterion, which conceptualizes the results of~\cite{BPP,BGP} and allows to compose maps while retaining the partially hyperbolic property. 

We apply this criterion in sections \ref{s.geodesic} and \ref{s.transverse} to prove Theorems \ref{t.geodesic} and \ref{t.transverse}, respectively. To be able to apply the criterion we need specific constructions for each case, using the ideas in \cite{BGP} to prove Theorem \ref{t.geodesic} and the ones of \cite{BZ} for Theorem \ref{t.transverse}. In addition, we show in subsection \ref{ss.obstruction} that not every mapping class can be realized as a partially hyperbolic diffeomorphism and in subsection \ref{ss.everymapping} that there are manifolds with infinite mapping class for which virtually every mapping class can be realized as a partially hyperbolic diffeomorphism. 

Finally, in section \ref{s.incoherence} we show dynamical incoherence of some of the examples presented here (proving Theorem \ref{t.incoh}), and in particular we show some relevant properties about the structure of the center curves of such examples. 

\section{A criterion for partial hyperbolicity}\label{s.criterion}

We provide a general criterion for partial hyperbolicity based on cone-fields. This criterion is obtained by extracting and formalizing 
the key point of the argument
 in the previous 
paper~\cite{BGP} (see also~\cite{HP-survey}).  The criterion has no restrictions on the dimensions of subbundles.

\begin{defi}\label{defi-htran} Let $f: M\to M$ and $g:N\to N$ be partially hyperbolic diffeomorphisms and let $h\colon M\to N$ be a diffeomorphism.
We say that $f$ is $h$-\emph{transverse} to $g$ and use the notation
$$f \ \xrightarrow[h]{} \ g$$ 
if $Dh(E^{uu}_f)$ is transverse to $E^{cs}_g$ and $Dh^{-1}(E^{ss}_g)$ is transverse to $E^{cu}_f$.  
\end{defi}

\begin{rema}
    It is easy to see that $f \ \xrightarrow[h]{} \ g$
    holds if and only if
    $g^{-1} \ \xrightarrow[{h^{-1}}]{} \ f^{-1}$. Further, if $f \
\xrightarrow[h]{} \ g$ then $f \ \xrightarrow[g^lhf^k]{} \ g$ for any
non-negative integers $k,l$.
\end{rema}

The $h$-transversality property is transitive in the following sense.

\begin{lemm}\label{l.h-transverse} Let $f_i:M_i \to M_i$, $i\in\{1, \dots,\ell\}$ be a family of partially hyperbolic diffeomorphisms and let $h_j\colon M_j\to M_{j+1}$, 
$j\in\{1,\dots, \ell-1\}$ be a family of diffeomorphisms such that $f_i$ is $h_i$-transverse to $f_{i+1}$.  Then, there exist positive  $m_2,\dots,m_{\ell-1}$ such that, 
for every $n_i\geq m_i$ 
$$f_1 \ \xrightarrow[h_{\bar n}]{}  \ f_\ell,$$
 where $h_{\bar n}= h_\ell \circ f_{\ell-1}^{n_{\ell-1}}\circ h_{\ell-1}\circ\cdots\circ f_2^{n_2}\circ h_1$. 
\end{lemm}

\begin{proof} We will show that $$ f_1 \ \xrightarrow[h_1]{} \ f_2 \ \xrightarrow[h_2]{} \ f_3,$$ 
implies that for large $n$ one has that $f_1 \ \xrightarrow[\hat h]{} \ f_3$ with $\hat h = h_2 \circ f_2^n \circ h_1$ as the statement follows by several applications of this. 

As $Dh_1(E^{uu}_{f_1})$ is transverse to $E^{cs}_{f_2}$ it follows that $Df_2^n(Dh_1(E^{uu}_{f_1}))$ converges uniformly to $E^{uu}_{f_2}$. The fact that $Dh_2(E^{uu}_{f_2})$ is transverse to $E^{cs}_{f_3}$ allows to conclude that $D(h_2 \circ f_2^n \circ h_1)(E^{uu}_{f_1})$ is transverse to $E^{cs}_{f_3}$. 

A symmetric arguments gives the other condition. 
\end{proof}

Now, we can formulate the criteria for partial hyperbolicity:

\begin{prop}\label{prop-criterion}
Let $f : M \to M$ be an absolutely partially hyperbolic diffeomorphism and $h: M \to M$ a diffeomorphism of $M$ such that $f \ \xrightarrow[h]{} \ f$. Then, there exists $N>0$ such that $f^n \circ h$ is absolutely partially hyperbolic for any $n\geq N$. 
\end{prop}

\begin{proof} We prove that $f^n \circ h$ has a weak absolutely partially hyperbolic
    splitting of the form $\Ecs \oplus \Euu$.
    A similar proof shows that $f^n \circ h$ has a splitting of the form
    $\Ess \oplus \Ecu$,
    and together these show that $f^n \circ h$ is 
    absolutely partially hyperbolic. We will use the classical cone criteria, see e.g. \cite[Appendix B]{BDV}. 

    Up to replacing $f$ by an iterate,
    there are $1 < \gam < \lam$ such that
    \begin{math}
        \|Df v^c\| < \gam < \lam < \|Df v^u\|
    \end{math}
    for all unit vectors $v^c  \in  \Ecs_f$ and $v^u  \in  \Euu_f$.
    By properties of partial hyperbolicity,
    there is a cone family $\Coneuu \subof TM$ such that
    $Df(\Coneuu) \subof \Coneuu$,
    \[
        \bigcap_{n  \ge  0} Df^n(\Coneuu) = \Euu_f,
    \]
    and
    $\|Df v\|  \ge  \lam \|v\|$
    for all
    $v  \in  \Coneuu$.
    Since
    \[
        \bigcap_{n  \ge  0} D(h \circ f^n)(\Coneuu) = Dh(\Euu_f)
    \]
    is transverse to $\Ecs_f$, it follows that
    \begin{math}
        D(h \circ f^k)(\Coneuu) \cap \Ecs_f = 0
    \end{math}
    for some integer $k$.
    By replacing $\Coneuu$ with $Df^k(\Coneuu)$,
    we freely assume that
    $Dh(\Coneuu) \cap \Ecs = \{0\}$.

    For any non-zero vector $v  \in  \Coneuu$, the angle between
    $D(f^n \circ h)(v)$ and $\Euu_f$ tends to zero as $n  \to  +\infty$
    and so there is $N$ such that
    $D(f^n \circ h)(v)  \in  \Coneuu$
    for all $n > N$.
    By a compactness argument,
    this $N$ may be chosen independently of $v$.
    Further, there is a constant $a > 0$
    such that
    \begin{math}
        \|D(f^n \circ h) v\|  \ge  a \lam^{n} \|v\|
    \end{math}
    for all $n > N$ and $v  \in  \Coneuu$.
    If $n$ is sufficiently large, then $a \lam^{n} > 1.$

    By a similar argument (and possibly adjusting $a$ and $N$),
    there is a cone family $\Conecs$ such that if $v  \in  \Conecs$, then
    \[
        D(f^n \circ h) \inv v  \in  \Conecs
        \qandq
        \|D(f^n \circ h) \inv v\|  \ge  \frac{a}{\gam^{n}} \|v\|
    \]
    for all $n > N$.
    Note also that $\frac{\gam^n}{a} < a \lam^n$ for all large $n$.
    These conditions together imply that $f^n \circ h$
    has a weak absolutely partially hyperbolic splitting of the
    form $\Ecs \oplus \Euu$.
\end{proof}

Indeed, it follows from the proof of the proposition that we have: 

\begin{sch}\label{adendum.criteria} 
In the setting of Proposition \ref{prop-criterion}, 
given any sufficiently narrow cone field $\cC^{uu}$ (resp. $\cC^{ss}$) in $M$ around $E^{uu}$ 
(resp.  around  $Dh^{-1}(E^{ss})$), there exists $N$ such that $f^n\circ h$ 
preserves the cone fields for $n>N$. 
\end{sch}

\begin{rema}\label{r.bundles}
Notice that if one denotes $F_{n,m}= f^n \circ h \circ f^m$ it follows that $F_{n,m}$ is smoothly conjugate to $f^{n+m}\circ h$ via $f^{-m}$. It follows that for $n+m$ sufficiently large, $F_{n,m}$ is absolutely partially hyperbolic. Notice also, that if both $n$ and $m$ are large, one gets that the bundles of $F_{n,m}$ are as close as those of $f$ as one wants. 

\end{rema}

Using Lemma \ref{l.h-transverse} one gets. 

\begin{prop}\label{p.h-transverse} Let $f$ and $g$ be two partially hyperbolic diffeomorphisms on $M$ such that there is a $C^1$-continuous path 
$\{f_t\}_{t\in[0,1]}$ of partially hyperbolic diffeomorphisms so that $f_0=f$ and $f_1=g$.  Then there exists a diffeomorphism $h$ of $M$ such that 
$f$ is $h$-transverse to $g$. 

Moreover:
\begin{itemize}
\item  $h$ is isotopic to $f^N$ for some large $N$. In particular,
if $f$ (and therefore $g$) is isotopic to identity, then $h$ can be chosen to be isotopic to identity.
\item If all the diffeomorphisms in the path $\{f_t\}$ preserve a smooth volume $\mathrm{vol}_t$ which varies smoothly with $t$, then, one can choose $h$ so that it maps $\mathrm{vol}_0$ to $\mathrm{vol}_1$. 
\end{itemize} 
\end{prop}
\begin{proof} First notice that, by continuity of the bundles with respect to $t$, and by compactness of the parameters set $[0,1]$, there is $\varepsilon>0$ such
that for any $s,t\in [0,1]$ with $|t-s|\leq \varepsilon$ one has 
$$f_t \ \xrightarrow[\mathrm{id}]{} \ f_s.$$ 
Then, one applies the above transitivity lemma to obtain $h$-transversality. The fact that $h$ is isotopic to a power of $f$ is immediate from the lemma and the fact that all $f_t$ are isotopic to $f$. 

To see that one can choose $h$ to preserve volume, notice that one can find, using Moser's trick (see \cite[Theorem 5.1.27]{KH}), a path $\varphi_t$ of diffeomorphisms sending $\mathrm{vol}_t$ into $\mathrm{vol}_0$. If follows that the diffeomorphisms $\{ \varphi_t \circ f_t \circ \varphi_t^{-1}\}$ form a smooth path of diffeomorphisms which preserve $\mathrm{vol}_0$. So applying the previous reasoning to the path $\varphi_t \circ f_t \circ \varphi_t^{-1}$ which connects $f$ and $\hat g=\varphi_1 \circ g \circ \varphi_1^{-1}$ one obtains that $f$ is $\hat h$-transverse to $\hat g$ with $\hat h$ a $\mathrm{vol}_0$ preserving diffeomorphism. Because $\varphi_1$ conjugates $g$ and $\hat g$ we see that $h= \varphi_1^{-1} \circ \hat h$ is the posited diffeomorphism. 
\end{proof}

\begin{rema} 
We have worked in the specific setting that is needed for this paper. It is not hard to see that the concept of $h$-transversality can be generalized to give general statements about \emph{dominated splittings}, \emph{normal hyperbolicity}, etc. Moreover, in Proposition \ref{prop-criterion} one just need to ensure transversality of certain flags and domination between bundles; then, to obtain partial (or absolute partial, or $r$-normal) hyperbolicity, it is enough to have the property for just one of the diffeomorphisms involved. We leave these statements to the reader as they will not be used here. 
\end{rema}

\section{Geodesic flows on higher genus surfaces}\label{s.geodesic}
In this section we prove Theorem~\ref{t.geodesic}.
Let us first rewrite and reinterpret  in our language a key proposition of the proof of \cite[Theorems 2.8 and 2.10]{BGP}.
Recall that, if $\gamma$ is a simple closed curve on a surface $S$ then $\tau_\gamma$ denotes the element of $\cM(S)$ containing the Dehn twist along $\gamma$. 

\begin{prop}\label{p.BGP-geodesic} Consider a closed oriented surface $S$ and let $\gamma$ be the free homotopy class
of a simple closed curve in $S$. Then there exist
\begin{itemize}\item a hyperbolic metric $\mu$  on $S$ and
 \item  a volume preserving diffeomorphism
$h\colon T^1S\to T^1 S$ in the mapping class $\iota(\tau_\gamma)$,
\end{itemize} such that 
the time one map $f$ of the geodesic flow of $(S,\mu)$ is $h$-transverse to itself. 
\end{prop}

\begin{proof}  
This follows from \cite[Proposition 2.9]{BGP}. In that proposition it is shown that one can choose a sequence of metrics $\mu_n$ in $S$ and a sequence of maps $h_n=D\rho_n$ in the class $\iota(\tau_\gamma)$ such that the bundles of the geodesic flow on $(S,\mu_n)$ are mapped by $h_n$ making arbitrarily small angle as $n\to\infty$ with themselves which implies $h_n$-transversality for a sufficiently large $n$ (this is called \emph{ph-respectful} in \cite{BGP}).


In \cite[Theorem 2.10]{BGP} it is explained how to modify the map to make it volume preserving using local coordinates around $\gamma$. 
\end{proof}

\begin{proof}[Proof of Theorem~\ref{t.geodesic}] Let $\varphi\in \cM(S)$. According to Dehn-Lickorish theorem (see for instance \cite[Theorem 4.1]{FaMa})  one can write 
$\varphi$ as a composition $\varphi=\tau_{\gamma_\ell}^{\varepsilon_\ell}\circ\cdots\circ \tau_{\gamma_1}^{\varepsilon_1}$,
where $\gamma_i$ are simple closed curves 
(not necessary disjoint nor distinct) and $\varepsilon_i\in\{-1,1\}$.

Proposition~\ref{p.BGP-geodesic}  yields hyperbolic metrics  $\mu_1, \dots,\mu_\ell$ on $S$ associated to $\gamma_i$, $i=1,\ldots \ell$, and volume preserving diffeomorphisms
$h_i\colon T^1S\to T^1S$, $i=1,\ldots \ell$, such that each time one map $f_i$ of the geodesic flow for $\mu_i$ is $h_i$-transverse to itself and $h_i$ is in the mapping class of $\iota(\tau_{\gamma_i}^{\eps_i})$. 

Recall that the Teichm\"uller space of hyperbolic metrics on $S$ is diffeomorphic to an open ball and, hence, is path connected (see \eg~\cite[Chapter 10]{FaMa}). Thus $f_i$ is isotopic to $f_{i+1}$ through a path of
time-one maps of geodesic flows which correspond to a path of hyperbolic metrics connecting $\mu_i$ to $\mu_{i+1}$. Clearly the path connecting $f_i$ to $f_{i+1}$ consists of volume preserving partially hyperbolic diffeomorphisms
(the preserved volume varying smoothly with the parameter). 

By Proposition~\ref{p.h-transverse}, we obtain diffeomorphisms $g_i\colon T^1S\to T^1S$ 
such that:
\begin{itemize}
\item $f_i$ is $g_i$-transverse to $f_{i+1}$,
\item $g_i$ sends the Liouville volume of 
$\mu_i$ to that of $\mu_{i+1}$, for $i\in \ZZ/\ell\ZZ$, 
\item $g_i$ is isotopic to identity.\end{itemize} 

Therefore we have

\[ \noindent  f_1 \ \xrightarrow[h_1]{} \ f_1 \ \xrightarrow[g_1]{} \ f_2 \ \xrightarrow[h_2]{} \ f_2 \ \xrightarrow[g_2]{}\cdots  \xrightarrow[g_{\ell-1}]{}  \ f_\ell \ \xrightarrow[h_\ell]{} \ f_\ell \ \xrightarrow[g_\ell]{} \ f_1  \]

Now, according to Proposition \ref{prop-criterion} (and Lemma \ref{l.h-transverse}), for every large $m$ the diffeomorphism
$$F=f_1^m\circ g_\ell\circ f_\ell^m\circ h_\ell \circ \cdots\circ h_2\circ f_2^m\circ g_1\circ f_1^m\circ h_1$$
is absolutely partially hyperbolic. By construction it is volume preserving. Finally, as the $g_i$ and the $f_i$ are isotopic to identity
the isotopy class of $F$ is the same as that of $h_\ell \circ\cdots\circ h_1$. Hence the mapping class of $F$ is given by
$\iota(\tau_{\gamma_\ell}^{\eps_\ell})\circ\cdots\circ\iota(\tau_{\gamma_1}^{\eps_1})=\iota(\varphi)$ as $\iota$ is a group homomorphism. 

Finally, according to \cite{BMVW,HHU}, there exists a small perturbation of $F$ which is stably ergodic, which completes  the proof.
\end{proof}

\begin{rema}\label{rema-metric}
The proof can be adapted to give that for any given hyperbolic metric $\mu$ with geodesic flow $\cG_t$ and $\varphi \in \cM(S)$ there exists a volume preserving diffeomorphism $h$ in $\iota(\varphi)$ so that $\cG_t$ is $h$ transverse to itself. 
\end{rema}

\begin{rema}\label{rema-circle} If we choose a reducible mapping class, \ie if it leaves invariant a simple closed geodesic, then one can apply the same reasoning as in \cite[Proposition 2.12]{BGP} to make a robustly transitive and stably ergodic perturbation of the example. Robust transitivity can also be obtained in other contexts as we will explain in Subsection \ref{ss.minimal}.
\end{rema}

\begin{rema}
Every closed surface $S$ admits a metric of constant negative curvature with an orientation reversing isometry, and this lifts to an $h$-transversality for the geodesic flow of the metric with itself. This allows us to extend Theorem \ref{t.geodesic}  to the case of an orientation reversing mapping classes $\tau$ of $\cM(S)$. Note that in this case $\iota(\tau)$ is orientation preserving because the derivative reverses orientation on the fibers.
\end{rema}

\begin{rema} It was noted to us by Livio Flamino that the above proof also admits the following variant. Equip $S$ with hyperbolic metrics $\mu_1$ and $\mu_2$ such that $\phi:(S,\mu_1)\to (S,\mu_2)$ is an isometry. Clearly $\phi$ defines an $h$-transversality for corresponding geodesic flows. Then connect $\mu_1$ and $\mu_2$ by a path of hyperbolic metrics. Then similarly to the proof above, by subdividing the path into small intervals and using $id\colon S\to S$ as $h$-transversalities, one realizes the mapping class of $\varphi$ by a partially hyperbolic diffeomorphism.
\end{rema}

%
%
\subsection{Some obstructions}\label{ss.obstruction} 

Recall (see~\cite[Chapter 25]{Johannson}) that the mapping class group of $T^1S$ fits into the following short exact sequence.

$$ 1 \to H_1(S,\ZZ) \to \cM(T^1S) \to \cM(S) \to 1 .$$
Theorem \ref{t.geodesic} implies that the set of mapping classes of $T^1S$ which are realized by partially hyperbolic diffeomorphisms surjects onto the whole $\cM(S)$. Here we show the following.

\begin{teo}
If a diffeomorphism $f \colon T^1S \to T^1S$ is partially hyperbolic and its mapping class projects to the identity in $\cM(S)$, then $f$ is isotopic to identity.
\end{teo}

The proof assumes familiarity with branching foliations~\cite{BI} and uses a result from~\cite{HaPS}. 

\begin{proof}
    Let $z$ be the element of $\pi_1(T^1S)$ given by a path around a fiber of
    the circle fibering.
    Equivalently, $z$ is a generator for the cyclic subgroup which is the center of $\pi_1(T^1S)$.
    Every element of $\pi_1(T^1S)$ may then be written in the form $[\gam'] \cdot z^k$
    where $k$ is an integer and
    the curve $\gam'$ in $T^1S$ is the derivative of a closed geodesic $\gam$
    in $S$.
    Up to replacing $f$ with $f^2$, we may assume $f_* (z) = z$ and so if $f_*$ is not the identity,
    there is a closed geodesic $\gam$ such that $f_*([\gam']) = [\gam'] \cdot z^k$ for some non-zero $k$.

    By \cite[Section 5.2]{HaPS} we know that the center-stable branching foliation $\cF^{cs}_{bran}$ for $f$ is horizontal.
    To be precise, we may first approximate $\cF^{cs}_{bran}$ by a true foliation
    $\cF_{\epsilon}$ where every leaf of the true foliation is isotopic to a leaf of the
    branching foliation and vice versa.
    Then, there exists a smooth conjugacy which is isotopic to identity and which puts $\cF_{\epsilon}$ transverse
    to the fibers of the circle fibering.
    Using a semiconjugacy result of Matsumoto \cite{Matsumoto},
    one may further show the following:
    for every leaf $\mathcal{L}$ of the weak stable foliation of the geodesic
    flow on $T^1S$,
    there is a leaf $L$ of $\cF^{cs}_{bran}$ isotopic to $\mathcal{L}$
    and vice versa.
    In particular, one of the leaves $L$ of $\cF^{cs}_{bran}$
    is isotopic to the weak stable leaf containing $\gam'$ and this implies that $L$
    itself contains a closed curve isotopic to $\gam'$.

    Both the leaf $L$ and its image $f(L)$ are cylinders
    and $\pi_1(f(L))$ is the cyclic subgroup of $\pi_1(T^1S)$
    generated by $f_*([\gam']) = [\gam'] \cdot z^k$.
    However, the semiconjugacy implies that $f(L)$ is isotopic
    to a leaf of the weak stable foliation of the geodesic flow.
    Since no such leaf has such a fundamental group, we arrive at a
    contradiction.
\end{proof}

\begin{coro}
The set of mapping classes realized by partially hyperbolic diffeomorphisms does not form a subgroup of $\cM(T^1S)$. 
\end{coro}
\begin{proof}
    It is enough to show that one can create different mapping classes projecting to the same mapping class in $\cM(S)$. Then, by composing one with the inverse of the other one obtains a non-trivial mapping class projecting to the identity, which would yield a contradiction. To see that one has examples in such mapping classes one just has to observe that Theorem~\ref{t.geodesic} produces mapping classes projecting into all of $\cM(S)$ and that the mapping classes which can be realized by partially hyperbolic diffeomorphisms are closed under conjugacy. For example, one can take two homologically non-trivial curves $\alpha$ and $\beta$ in $S$ whose intersection number is one. Then denote by $f_\alpha\colon T^1S\to T^1S$ a partially hyperbolic diffeomorphism in the mapping class of the (differential of) the Dehn twist along $\alpha$ and denote by $v_\beta\colon T^1S\to T^1S$ a diffeomorphism whose mapping class projects to the identity in $\cM(S)$ but for which $\beta'$ is not homotopic to $v_\beta(\beta')$.
    Then $f_\alpha$ and $v_\beta\circ f_\alpha\circ v_\beta^{-1}$ are partially hyperbolic diffeomorphisms from different mapping classes which project to the same mapping class in $\cM(S)$.
\end{proof}

\section{Anosov flows transverse to tori}\label{s.transverse}
In this section we prove Theorem \ref{t.transverse}.

\subsection{Dehn twists in three dimensions}\label{ss-dehn}

Here, we introduce a three-dimensional analogue
of two-dimensional Dehn twists.
Consider a two-dimensional torus $T$ embedded in a
three-dimensional manifold $M.$
Let $U$ be a tubular neighbourhood of this torus.
Then we may realize $U$ as the image of an embedding
\begin{math}
    i : \mathbb{T}^2 \times (-1,+1) \to M
\end{math}
where $\mathbb{T}^2$ is the standard 2-torus given
by the quotient $\mathbb{R}^2 / \mathbb{Z}^2.$
Further assume that $i$ maps $\mathbb{T}^2 \times \{0\}$ to $T.$
Consider now a curve $\gamma$ embedded in $T.$
After a homotopy, we can reduce to the case where
$\gam$ is the image of a linear curve in $\mathbb{T}^2.$
That is, there is a vector $(p,q) \in \mathbb{Z}^2$
such that $\gamma$ is the image under $i$ of the set
\[
    \bigg\{ (t p, t q, 0) : t \in [0,1] \bigg\}
    \subof \mathbb{T}^2 \times (-1,+1)
\]
We now define
\emph{the Dehn twist along $T$ directed by $\gamma$}.
This will be an element of the mapping class group of $M$
and so it is enough to define a diffeomorphism
$\tau : M \to M$
representing this element.
In fact, this map will be the identity outside of $U$
and so we first define a Dehn twist on $\mathbb{T}^2 \times (-1,+1).$
Let $h : (-1,+1) \to [0,1]$ be a smooth bump function
such that $h(z) = 0$ for $z < - \tfrac{1}{2}$
and $h(z) = 1$ for $z > \tfrac{1}{2}.$
Then define $\phi : \mathbb{T}^2 \times (-1,+1) \to \mathbb{T}^2 \times (-1,+1)$
by
\[
    \phi(x,y,z) = (x + p h(z), y + q h(z), z).
\]
Finally, define $\tau : M \to M$ by
\begin{math}
    \tau(v) = i \circ \phi \circ i \inv(v)
\end{math}
if $v \in U,$ and $\tau(v) = v$ otherwise.

\subsection{Anosov flows transverse to tori}\label{ss.covering}
Let $X$ be an Anosov vector field on a 3-manifold $M$ and let $T\subset M$ be a torus transverse to $X$. It is well known that $T$ must be incompressible~\cite{Br}. A systematic study of Anosov flows transverse to tori has been recently carried out in~\cite{BBY}, yet some questions still remain open. 

\subsection{Connecting different vector fields} 

The following proposition is immediate. We state it to emphasize the parallel between this case and the proof of Theorem \ref{t.geodesic} where a analogous statement was obtained using the connectedness of the Teichm\"uller space. 

\begin{prop}\label{prop.isotopicfunct} 
Let $X$ be a smooth Anosov vector field on $M$ and let $\rho_1, \rho_2: M \to (0,1] $ be smooth functions. Then, the time one maps of the flows generated by the vector fields $\rho_1 X$ and $\rho_2 X$ are isotopic through partially hyperbolic diffeomorphisms. 
\end{prop}

As an immediate consequence of this proposition and Proposition \ref{p.h-transverse} we have the following.

\begin{coro}\label{coro-htransv}
Let $X$ be a smooth Anosov vector field on $M$ and let $\rho_1, \rho_2: M \to (0,1]$ be smooth functions. Then, there exists a diffeomorphism $g: M \to M$ such that the time one map of the flow of $\rho_1 X$ is $g$-transverse to the time one map of the flow of $\rho_2X$.  Moreover, $g$ is isotopic to identity.
\end{coro}

\begin{rema}
If $X$ is volume preserving, then, again by Proposition~\ref{p.h-transverse},  one can choose $g$ so that it sends the volume preserved by $\rho_1X$ to the volume preserved by $\rho_2 X$. 
\end{rema}

\subsection{Coordinates in flow boxes}

We consider an Anosov vector field $X$ transverse to a torus $T$.  Without loss of generality we can assume that we have chosen a metric in $M$ such that $X$ is everywhere orthogonal to $T$ and such that $\|X(x)\|=1$ for all $x \in M$. 

Since $X$ is transverse to $T$ we obtain that the foliations $W^{cs}_{X}$ and $W^{cu}_{X}$ induce (transverse) foliations $F^{s}_{X}$ and $F^{u}_{X}$ on $T$ by taking the intersections.

We can consider coordinates $\theta_T\colon T \to \TT^2$ where $\TT^2= \RR^2 /_{\ZZ^2}$ is equipped with the
usual $(x,y)$-coordinates ($mod \ 1$). 
We denote by $F^s$ and $F^u$ the foliations $\theta_T(F^{s}_{X})$ and $\theta_T(F^{u}_{X})$, respectively.

For each $\eta>0$ consider a manifold $\hat M_{\eta}$ diffeomorphic to $M$ with an Anosov vector field $\hat X^\eta$ obtained as follows: 

\begin{itemize}
\item first cut $M$ along $T$ to obtain a 3-manifold $M'$ with two boundary components $T_0$ and $T_1$ such that the vector field $X$ points outwards on $T_0$ and inwards on $T_1$,
\item glue in a flow box $\cU_\eta=[0,\eta] \times T$ to $M'$ so that $\{0\}\times T$ glues to $T_0$ and $\{\eta\}\times T$ glues to $T_1$, 
\item define $\hat X^\eta$ to be the unit horizontal vector field on $[0,\eta] \times T$ which glues well with $X$. 
\end{itemize}

Denote by $\hat T$  the torus $T_0 \cong \{0\} \times T$ in $\hat M_\eta$ which can be identified with $T$ in the obvious way. We will simply write $\hat X$ for $\hat X^\eta$ and we will write $\hat X_t$ for the time $t$ map of the flow generated by $\hat X$. It is straightforward to check that the flow generated by $\hat X$ is Anosov (even if the constants worsen as $\eta$ increases). The fact that $\hat M_\eta$ is diffeomorphic to $M$ is also immediate because $T$ has a neighborhood in $M$ diffeomorphic to $(0,1) \times T$. 

The following proposition will be helpful to choose nice deformations of the flow in order to apply Proposition \ref{prop-criterion} and Corollary \ref{coro-htransv}. 

\begin{prop}\label{prop.function}
There exists a smooth function $\rho_\eta: M \to (0,1]$ such that the vector field $\rho_\eta  X$ is smoothly conjugate to $\hat X^\eta$. 
\end{prop}

\begin{proof}
Pick a small $\eps>0$ and pick an increasing diffeomorphism $\psi\colon[-\eps,\eps +\eta]\to[-\eps,\eps]$ such that $\psi'\le 1$ and $\psi-x$ is $C^\infty$ flat at the endpoints. Let $T_\eps=\bigcup_{|t|\leq \eps} X_t(T)\subset M$. In a similar way, let $\hat T^\eta_\eps=\bigcup_{-\eps \leq t \leq \eta+\eps} \hat X_t (T_0)\subset \hat M_\eta$. Clearly these are smooth embeddings of $[-\eps,\eps] \times T$ and $ [-\eps,\eta+\eps] \times T$ in $M$ and $\hat M_\eta$, respectively, so that orbits of the flows correspond to horizontal lines. Now, choose a diffeomorphism from $\hat M_\eta$ to $M$ which is the identity outside $\hat T^\eta_\eps$ and identifies $\hat T^\eta_\eps$ with $T_\eps$ via $\psi\times id$. This diffeomorphism yields the conjugacy and the posited time change $\rho_\eta  X$. 
\end{proof}

Recall that the inserted flow box $\cU_\eta=[0,\eta]\times T\subset \hat M_\eta$ can be viewed a sweep-out of $T$ up to time $\eta$:

$$\cU_\eta = \bigcup_{0 \leq t \leq \eta}  \hat X_t(\hat T)$$

\noindent Consider the ``straightening diffeomorphism" $H_\eta: \cU_\eta \to [0,1] \times \TT^2$ given by

\begin{equation}\label{eq.HN}
 H_\eta(\hat X_t(p)) = \left(\frac{t}{\eta}, \theta_T(p)\right), p\in\hat T
\end{equation}

For fixed $\eta$ we will denote by $\hat W^{\sigma}_\eta$ and $\hat E^{\sigma}_\eta$ the corresponding foliations and invariant bundles for   $\hat X$ ($\sigma=cs,cu,ss,uu$). 

We also denote by $\cF^{uu}$ and $\cF^{ss}$  the one-dimensional foliations of $[0,1] \times \TT^2$ which in $\{t\} \times \TT^2$ coincide with the foliations $\{t\} \times F^u$ and $\{t\} \times F^s$, respectively.

\begin{lemm}\label{l.limit} The diffeomorphism $H_\eta$ has the following properties
\begin{itemize}
\item  $H_\eta(\hat W^{cs}_{\eta} \cap \cU_\eta) = [0,1] \times F^s\stackrel{\mathrm{def}}{=}
\cF^{cs}$

\item $H_\eta(\hat W^{cu}_{\eta} \cap \cU_\eta) = [0,1] \times F^u \stackrel{\mathrm{def}}{=}
\cF^{cu}$

\item $DH_\eta(\hat E^{ss}_{\eta})$ converges to the tangent
bundle of the foliation $\cF^{ss}$ as $\eta \to
\infty$.

\item $DH_\eta(\hat E^{uu}_{\eta})$ converges to the  tangent
bundle of the foliation $\cF^{uu}$ as $\eta \to
\infty$.

\end{itemize}
\end{lemm}

It is important to remark that the above convergence is with respect to the standard metric in $[0,1] \times \TT^2$ and not with respect to the push forward metric from the manifold via $H_\eta$.

\begin{proof} Because the differential of $H_\eta$ maps the vector field $\hat X$ to the vector field $\frac{1}{\eta} \frac{\partial}{\partial t}$ the first two properties follow. Note that the component of $\hat E^{\sigma}_{\eta}$ along $\hat X$ is uniformly bounded
($\sigma=ss,uu$). Therefore, contraction by a factor $\frac{1}{\eta}$ implies the posited limit behavior in the latter properties.
\end{proof}

\subsection{A diffeomorphism in a flow box which preserves transversalities}\label{ss-path}

Assume that there exists a smooth path $\{\varphi_s\}_{s\in [0,1]}$ of diffeomorphisms
of $\TT^2$ such that

\begin{itemize}
\item $\varphi_s=\mathrm{id}$ for $s$ in neighborhoods of $0$ and $1$,
\item the closed path $s\mapsto \varphi_s$ is not homotopically
trivial in $\mathrm{Diff}(\TT^2)$, \item for every $s \in [0,1]$:
$$\varphi_s(F^u)\pitchfork F^s.$$
\end{itemize}

We use the coordinate chart $H_\eta\colon \cU_\eta\to[0,1]\times\TT^2$ to define
 diffeomorphism $\cG_\eta : \cU_\eta \to \cU_\eta$ by
$$
(s,x,y) \mapsto (s, \varphi_s(x,y)).
$$

The following lemma is immediate from our choice of $\{\varphi_s\}_{s\in [0,1]}$.
\begin{lemm}\label{l.perturbationorbitspace}
The diffeomorphism $\cG_\eta$ has the following properties:
\begin{itemize}
\item $\cG_\eta(H_\eta^{-1}(\cF^{uu}))$ is
transverse to $\hat W^{cs}_{\hat X} = H_\eta^{-1}(\cF^{cs})$,
\item $\cG_\eta(\hat W^{cu}_{\hat X})=\cG_\eta(H_\eta^{-1}(\cF^{cu}))$ is transverse
to $H^{-1}_\eta(\cF^{ss})$.
\end{itemize}
\end{lemm}

Let $h_\eta \colon \hat M_{\eta} \to \hat M_{\eta}$ be the diffeomorphism defined as 

\begin{equation}\label{eq:dehnhN}
h_\eta(x)= \left\{ \begin{array}{cc}
  x & \text{if } \ x \notin \cU_\eta \\
  \cG_\eta(x) & \text{if } \ x \in \cU_\eta \\
\end{array}\right.
\end{equation}

\noindent If $\gamma$ denotes the element of $\pi_1(T)$ associated to the loop $\{\theta_T^{-1}\circ \varphi_t \circ \theta_T(x)\}_t$, then, by construction, $h_\eta$ is a diffeomorphism that belongs to $\tau(T,\gamma)$ \ie a Dehn twist along $\gamma$ (cf. \S \ref{ss-dehn}). 

\begin{prop}\label{c.transversalityBL} There exists $\eta_0>0$ such that for all $\eta \geq \eta_0$:
\begin{itemize}
\item $Dh_\eta(E^{uu}_{\hat X})$ is transverse to $E^{cs}_{\hat X}$,
\item $Dh_\eta(E^{cu}_{\hat X})$ is transverse to $E^{ss}_{\hat X}$.
\end{itemize}
\end{prop}

\begin{proof} This follows by combining Lemma \ref{l.perturbationorbitspace}, Lemma~\ref{l.limit} and the fact that outside $\cU_\eta$ the diffeomorphism $h_\eta$ is the identity.
\end{proof}

As a consequence, using Proposition \ref{prop.function}, we have the following.

\begin{coro}\label{c.flowhtransverse}
Given $X$ an Anosov flow on $M$ transverse to an embedded torus $T$ and $\gamma \in G(T,X)$, there exists a smooth function $\rho_T: M \to (0,1]$  and a diffeomorphism $h_\gamma:M \to M$ in the mapping class $\tau(T,\gamma)$ such that if $f_T$ is the time-one map of the flow generated by $\rho_T X$, then $f_T$ is $h_\gamma$-transverse to itself. 
\end{coro}

\begin{rema}\label{r.volumepreserving}
Notice that if there is a volume form $\omega$ on $T$ such that $\varphi_s$ preserves the form $(\theta_T)_\ast (\omega)$ for every $s$, then $h_\eta$ preserves volume form $\omega \wedge dX$.
\end{rema}

\subsection{Proof of Theorem \ref{t.transverse}.}

Let $X$ be an Anosov flow on $M$ and let $T_1, \ldots, T_\ell$ be tori (not necessarily disjoint or distinct) transverse to $X$. Choose $\gamma_i \in G(T_i,X)$. 

Using Corollary \ref{coro-htransv} and Corollary \ref{c.flowhtransverse} one obtains diffeomorphisms $h_i$ and $g_i$ of $M$ such that if $f_i$ denotes the time one map of the flow generated by $\rho_{T_i}X$, then we have:

\[ \noindent f_1 \ \xrightarrow[h_1]{} \ f_1 \ \xrightarrow[g_1]{} \ f_2 \ \xrightarrow[h_2]{} \ f_2 \ \xrightarrow[g_2]{}\cdots  \xrightarrow[g_{\ell-1}]{} \ f_\ell \ \xrightarrow[h_\ell]{} \ f_\ell \ \xrightarrow[g_\ell]{} \ f_1  \]

\noindent with $g_i$ isotopic to identity and $h_i \in \tau(T_i,\gamma_i)$. 

Now, according to  Proposition \ref{prop-criterion} (and Lemma \ref{l.h-transverse}), for every large $m$ the diffeomorphism
$$F=f_1^m\circ g_\ell\circ f_\ell^m\circ h_\ell \circ \cdots\circ h_2\circ f_2^m\circ g_1\circ f_1^m\circ h_1$$
is absolutely partially hyperbolic. Clearly, the homotopy class of $F$ is the composition of $\tau(T_\ell,\gamma_\ell) \circ \ldots \circ \tau(T_1,\gamma_1)$. 
 $\hfill\square$
 
 \subsection{Mapping class groups of 3-manifolds} 
 In order to proceed with the proof of Corollary~\ref{corollary}, we need to briefly recall the description of the mapping class group $\cM(M)$ of a Haken manifold $M$. This description relies on Johannson's characteristic submanifold theory which we briefly summarize. For full generality and full details we refer to~\cite[Chapters VI and IX]{Johannson}; see also~\cite[\S 4]{McC} and~\cite[Chapter 4]{CM}.
 
 Let $M$ be a closed ($\partial M=\varnothing$) Haken 3-manifold. Then the {\it characteristic submanifold} $\Sigma\subset M$ is a codimension 0 submanifold of $M$ which consists of all Seifert-fibered pieces of the JSJ decomposition of $M$. A torus which separates two Seifert-fibered pieces of the JSJ decomposition yields two torus boundary components of $\Sigma$ so that the number of connected components of $\Sigma$ is precisely the number of Seifert-fibered pieces in the JSJ decomposition. 
 Johannson's Classification Theorem asserts that every diffeomorphism  of $M$ is homotopic to a homeomorphism which preserves the characteristic submanifold $\Sigma$. Consider the subgroup $\cM(M,\Sigma)$ of $\cM(M)$ given by diffeomorphisms whose restriction to the closure of the complement of $\Sigma$ is isotopic to identity. By work of Johannson,  $\cM(M,\Sigma)$ has finite index in $\cM(M)$.
(In the proof of Corollary~\ref{corollary} below, the fact that $\cM(M,\Sigma)$ has finite index will be immediate from Johannson's Classification Theorem.)
 
 Restricting to connected components $\Sigma_i$, $i=1,\ldots n$, of $\Sigma$ yields an exact sequence (see~\cite[\S 4]{McC})
 \begin{equation}
 \label{eq_K}
 1\to K\to \cM(M,\Sigma)\to \prod_{i=1}^n\bar\cM(\Sigma_i,\partial \Sigma_i)\to 1
 \end{equation}
 where $K$ consists of mapping classes which are represented by homeomorphisms $h\colon M\to M$ which are identity outside a small product neighborhood of $\partial \Sigma$ and $\bar\cM(\Sigma_i,\partial \Sigma_i)$ is the subgroup of the full mapping class group $\cM(\Sigma_i,\partial \Sigma_i)$ which given by homeomorphisms $h\colon \Sigma_i\to\Sigma_i$ whose restriction to $\partial\Sigma_i$ is isotopic to identity.
 
 \begin{prop}
 \label{prop_K}
 Let $M$ be a Haken manifold admitting an Anosov flow $X$ which is transverse to the boundary $\partial\Sigma$ of the characteristic submanifold $\Sigma=\cup_{i=1}^n\Sigma_i$. Assume that for each torus boundary component $T$ of $\Sigma$ the group of permitted Dehn twists $G(T,X)$ is maximal, \ie isomorphic to $\ZZ^2$. Also assume that the mapping class groups of the Seifert fibered pieces $\bar\cM(\Sigma_i,\partial \Sigma_i)$, $i=1,\ldots n$, are finite. Then the mapping class group $\cM(M)$ has a finite index subgroup $K$ such that each mapping class in $K$ can be represented by an absolutely partially hyperbolic diffeomorphism.
 \end{prop}
 \begin{proof}
 By the above discussion and the assumption on $\bar\cM(\Sigma_i,\partial \Sigma_i)$, $i=1,\ldots n$, the subgroup $K$ from~(\ref{eq_K}) has finite index in $\cM(M)$. Each element of $K$ is represented by a Dehn twist along $\partial \Sigma$. Hence, by the assumption on $G(T,X)$, we have $K\subset \mathcal T(X)$ (recall that $\mathcal T(X)$ is the group generated by permitted Dehn twists). Hence each mapping class from $K$ is represented by a partially hyperbolic diffeomorphism according to Theorem~\ref{t.transverse}.
 \end{proof}
 
 In order to demonstrate examples to which Proposition~\ref{prop_K} applies we need to have Seifert-fibered manifolds $\Sigma_i$ with finite $\bar\cM(\Sigma_i,\partial \Sigma_i)$. Thus we need to recall Johannson's description of the mapping class group of Seifert-fibered manifolds with boundary~\cite[\S 25]{Johannson}.
 
 Let $\Sigma$ be a Seifert-fibered manifold with a non-empty boundary $\partial\Sigma$. Consider the mapping class group $\cM(\Sigma,\partial\Sigma)$ of self homeomorphisms $h\colon(\Sigma,\partial\Sigma)\to (\Sigma,\partial\Sigma)$ modulo boundary preserving isotopies ($h$ is allowed to permute boundary components). If one excludes six exceptional manifolds~\cite[5.1.1-5.1.6]{Johannson} (which we will not encounter here) then homeomorphism $h$ is isotopic to a fiber preserving homeomorphism~\cite[Corollary 5.9]{Johannson}. It follows~\cite[Proposition 25.3]{Johannson} that the mapping class group fits into a split short exact sequence
 \begin{equation}
 \label{eq_exact}
 1\to \cM^0(\Sigma,\partial\Sigma)\to \cM(\Sigma,\partial\Sigma)\to\cM(B,\partial B)\to 1,
\end{equation}
where $\cM^0(\Sigma,\partial\Sigma)$ is the subgroup of the mapping class group given by homeomorphisms which preserve each fiber (\ie homeomorphisms which fiber over identity) and $\cM(B,\partial B)$ is the mapping class group of the underlying orbifold $B$. (In the setup to which we are about to specify in order to obtain Corollary~\ref{corollary}, manifold $\Sigma$ will be circle-fibered and, hence, $\cM(B,\partial B)$ will simply be the mapping class group of a surface $B$.) Further, $\cM^0(\Sigma,\partial\Sigma)$ is isomorphic to $H_1(B,\partial B)$ (relative first homology) and is generated by certain vertical Dehn twists~\cite[Lemma 25.2]{Johannson}. 

Recall that Johannson's work elements of $\cM(\Sigma,\partial\Sigma)$ are represented by fiber-preserving homeomorphisms. Further if $h_1, h_2\colon \Sigma\to\Sigma$ are isotopic fiber preserving homeomorphisms then by work of Waldhausen~\cite{Waldhausen} (see discussion on p. 85) $h_1$ is isotopic to $h_2$ via a path of fiber preserving homeomorphisms. It follows that the
short exact sequence ~(\ref{eq_exact}) restricts to a short exact sequence
 \begin{equation}
 \label{eq_exact2}
 1\to\bar \cM^0(\Sigma,\partial\Sigma)\to\bar \cM(\Sigma,\partial\Sigma)\to\bar\cM(B,\partial B)\to 1,
\end{equation}
where $\,\,\,\bar{}\,\,\,$ indicates the subgroup (of corresponding mapping class group) induced by homeomorphisms whose restriction to the boundary is isotopic to identity.
 
 \subsection{Proof of Corollary~\ref{corollary}}\label{ss.everymapping} Denote by $\Sigma$ the total space of the (unique) non-oriented circle bundle $S^1\to \Sigma\stackrel{p}{\to} B$, where $B$ is the projective plane $\RR P^2$ with two disks removed (see Figure~\ref{fig:1}). Then manifold $\Sigma$ is orientable and has two torus boundary components. Bonatti and Langevin exhibited a self gluing of $\Sigma$ (via an orientation-reversing diffeomorphism of the 2-torus) which yields a closed manifold $M$ which supports an Anosov flow $X$ transverse to the gluing torus~\cite{BL}. The closed manifold $M$ is not Seifert-fibered, hence the characteristic submanifold of $M$ is the complement of a product neighborhood of the gluing torus and, therefore, can be identified with $\Sigma$. Note that for $M$ the subgroup $K$ appearing in~(\ref{eq_K}) is the subgroup of Dehn twists along the gluing torus and, hence, is isomorphic to $\ZZ^2$. 
 Further by combining the description of invariant foliations near the gluing torus provided by~\cite{BL} with the description of permitted Dehn twists~\cite{BZ} we have that all Dehn twists are permitted. Hence Proposition~\ref{prop_K} applies and, according to (\ref{eq_exact2}), in order to complete the proof of Corollary~\ref{corollary} in the case when $n=1$ it only remains to show that the groups $ \bar\cM^0(\Sigma,\partial\Sigma)$ and $\bar \cM(B,\partial B)$ are finite. In fact, the form of the intersection of the foliations with the tori allow to see that one can make the Dehn-twists in a volume preserving way (see \cite[Lemma 3.8]{BGP}). 
 
 \begin{figure}[ht]
\vspace{-0.5cm}
\begin{center}
\includegraphics[scale=0.6]{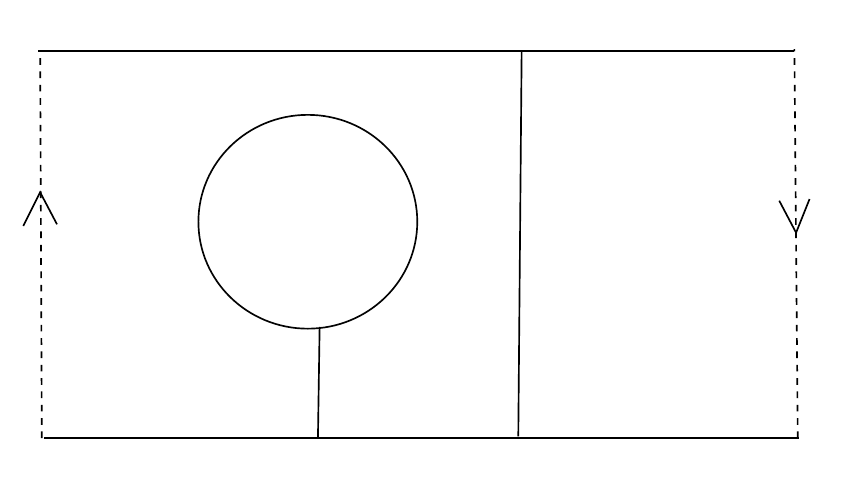}
\begin{picture}(0,0)
\put(-3,100){\small{$B$}}
\put(-95,70){\small{$h$}}
\put(-155,30){\small{$\alpha$}}
\end{picture}
\end{center}
\vspace{-0.5cm}
\caption{\small{The surface $B$.}\label{fig:1}}
\end{figure}

 The mapping class group  $\bar \cM(B,\partial B)$ can be easily seen to be $\ZZ/2\ZZ$.
 By~\cite[Lemma~25.2]{Johannson}, we have the isomorphism $\cM^0(\Sigma,\partial\Sigma)\simeq H_1(B,\partial B)$. The group $H_1(B,\partial B)$ is the first homology of $\RR P^2$ with two points identified. Hence $\cM^0(\Sigma,\partial\Sigma)\simeq H_1(B,\partial B)\simeq \ZZ\oplus\ZZ/2\ZZ$. Further,~\cite[ Proof of Case~2 of Lemma~25.2]{Johannson} gives an explicit description of generators of $\cM^0(\Sigma,\partial\Sigma)$ as follows. Let $\alpha$ be a simple curve which connects the boundary components of $B$ and let $h$ be a simple curve which connects a boundary component to itself as indicated in Figure \ref{fig:1}. Let $A_\alpha=p^{-1}\alpha$ and $A_h=p^{-1} h$ be corresponding annuli. Then, the Dehn twists along $A_\alpha$ and $A_h$ are the generators of $\cM^0(\Sigma,\partial\Sigma)$, the first one being the infinite order generator and the second one being the generator of order two. Clearly the Dehn twist along $A_\alpha$ restricts to a Dehn twist on each torus boundary component and, hence, is not isotopic to identity on the boundary. However, the restriction of Dehn twist along $A_h$ to the boundary $\partial \Sigma$ can be seen to be isotopic to identity (two Dehn twists on the boundary component given by $h$ are inverses of each other). We conclude that  $\bar\cM^0(\Sigma,\partial\Sigma)$ is $\ZZ/2\ZZ$ and that $K\simeq \ZZ^2$ indeed has a finite index in $\cM(M)$.
 
 To address the case $n>1$, consider finite cyclic covers $M_n\to M$ obtained by gluing $n$ copies of $\Sigma$ using the Bonatti-Langevin gluing. Then the lifted Anosov flow on $M_n$ is transverse to each of the gluing tori. The characteristic submanifold of $M_n$ is the disjoint union of $n$ copies of $\Sigma$. Then the same considerations as above apply to show that $K\simeq \ZZ^{2n}$ has finite index in $\cM(M_n)$. Further, because the group of Deck transformations $\ZZ/ n\ZZ$ of the covering $M_n\to M$ commutes with the Anosov  flow one can, in fact, realize a larger subgroup $K\oplus \ZZ/ n\ZZ$ by partially hyperbolic diffeomorphisms by following the arguments for the proof of Theorem~\ref{t.transverse}. More specifically, one has to utilize the fact that if $h$ denotes a Deck transformation of the covering $M_n\to M$ then the time one map of the Anosov flow is $h$-transverse to itself.
As this is
only needed to go from $K$ to $K\oplus \ZZ/ n\ZZ$,
we leave the details to the reader.

\section{Dynamical incoherence}\label{s.incoherence}
In this section, we consider the examples built in Theorem~\ref{t.geodesic} on the unit tangent bundle $T^1S$ of an oriented higher genus  surface $S$ which are 
isotopic to the projectivizations of pseudo-Anosov diffeomorphisms of $S$. 

Consider a hyperbolic surface $(S,\mu)$ and an element $\varphi$ of the mapping class group
$\cM(S)$. We denote by $\{\cG^t\}_{t\in\RR}$ the geodesic flow for the hyperbolic metric $\mu$ and by $Z$ the vector field which generates it.
Let $E^{ss}_Z\oplus \RR Z\oplus E^{uu}_Z$ be the hyperbolic splitting for the geodesic flow. 

According to Theorem~\ref{t.geodesic} and Remarks~\ref{r.bundles} and \ref{rema-metric}, there exists a diffeomorphism $f\colon T^1S\to T^1S$ such that 
\begin{itemize}
 \item $f$ belongs to the mapping class $\iota(\varphi)$;
 \item for any $t\geq 0$  the diffeomorphism $f_t= \cG^t\circ f\circ \cG^t$ is absolutely partially hyperbolic and volume preserving; 
 we denote by $E^{ss}_t\oplus E^c_t\oplus E^{uu}_t$ 
 the corresponding partially hyperbolic splitting.
 \item the subbundles $E^{ss}_t$, $E^c_t$ and $E^{uu}_t$ converge uniformly to $E^{ss}_Z$, $\RR Z$ and $E^{uu}_Z$, respectively, as $t\to +\infty$.
\end{itemize}

The goal of this section is to prove the following result (which is non-vacuous according  to the above discussion).
\begin{theo}\label{t.incoherent} Consider a hyperbolic surface $(S,\mu)$ and a pseudo-Anosov mapping class $\varphi\in\cM(S)$. Let $Z$ be the vector field which generates the geodesic flow on $T^1S$.

Then there exists $\eta>0$ such that, if $f\colon T^1S\to T^1S$ is a diffeomorphism satisfying the following conditions:
\begin{itemize}
\item$f$  is partially hyperbolic;
 \item $f$ belongs to the mapping class $\iota(\varphi)$;
 \item the subbundles of the partially hyperbolic splitting $E^{ss}\oplus E^c\oplus E^{uu}$ of $f$ are $\eta$-close to the 
 bundles of the hyperbolic splitting  $E^{ss}_Z\oplus \RR Z\oplus E^{uu}_Z$ of the geodesic flow;
\end{itemize}
 then $f$ is dynamically incoherent. 
 \end{theo}

In fact, we will prove a slightly more general result. To state the result we need to recall the Nielsen-Thurston classification of mapping classes of $S$~\cite{Th, HTh}. According to this classification, each homeomorphism of a compact oriented surface $S$ of negative Euler characteristic with (possibly empty) boundary $\partial S$ is isotopic to one of the following ``normal forms'' $h\colon S\to S$ 
\begin{itemize}
\item $h$ is {\it periodic}, that is, $h^n=id_S$ for some $n\ge 1$,
 \item $h$ is {\it pseudo-Anosov},
\item $h$ is {\it reducible}. Namely, there exists a finite collection of simple closed curves whose union is invariant under $h$. This collection partitions $S$ into a finite collection of subsurfaces each of which is 
invariant under $h^n$ for some $n\ge 1$. Moreover, the restriction of $h^n$ to each of the subsurfaces is either periodic or pseudo-Anosov.
\end{itemize}

\begin{add}
 Theorem~\ref{t.incoherent} also holds if instead of assuming that mapping class $\varphi$ is pseudo-Anosov we assume that $\varphi$ is a reducible mapping class with at least one pseudo-Anosov component.
\end{add}
 
From now on we fix the hyperbolic surface $(S,\mu)$ its geodesic flow $\cG^t$ and the generating vector field $Z$.

\subsection{Shadowing property of the geodesic flow $\cG^t$.}

Let us first state the standard shadowing lemma for Anosov flows (see \cite[Theorem 18.1.6]{KH}). Recall that a $\delta$-\emph{pseudo-orbit} for a flow generated by a vector field $X$ is a curve $c: \mathbb{R} \to M$ such that $\|c'(t) - X(c(t))\| \leq \delta$ for all $t \in \RR$. We say that a curve $c: \mathbb{R}\to M$ is $\eps$-\emph{shadowed} by an orbit of $X$ if there exists $x\in M$ and a reparametrization $s: \mathbb{R} \to \mathbb{R}$  such that $|s'(t)-1| \leq \eps$ for all $t$ and such that $d(c(s(t)), X_t(x)) \leq \eps$ for all $t\in \mathbb{R}$. 

\begin{theo}\label{t.shadow}
Let $X$ be an Anosov flow on a closed manifold $M$. Then, for every $\eps>0$ there exists $\delta>0$ such that every $\delta$-pseudo orbit is $\eps$-shadowed by an orbit of $X$. Moreover, if $\eps$ is sufficiently small, then the shadowing orbit is unique. 
\end{theo}

In our setting, we also obtain shadowing injectivity as follows. We can parametrise center leaves as $(c, \vec c): \RR \to T^1S$ and the notation is to indicate that the projection into $S$ is the curve $c: \RR \to S$.

\begin{theo}\label{t.uniqueshadowing} 
 Let $Z$ be the vector field which generates the geodesic flow $\cG^t$ of a hyperbolic surface $S$. Then there is $\varepsilon_0>0$ such that for every
$\varepsilon\in(0,\varepsilon_0)$ there is $\delta>0$ with the following property: 
\begin{itemize}
\item consider  any dynamically coherent partially hyperbolic diffeomorphism $f\colon T^1 S\to T^1S$  whose subbundles $E^s_f,E^c_f$ and $E^u_f$ of its partially hyperbolic splitting are 
 $\delta$-close to the corresponding ones of $Z$, and let $X_f^c$ a vector field orienting $E^c_f$ so that $X_f^c$ is $\delta$-close to $Z$, 
\item consider  any pair of different center leaves $(c_1, \vec c_1), (c_2, \vec c_2) \colon \mathbb{R} \to T^1S$ tangent to $X_f^c$, 
\end{itemize}
 \noindent then the orbits of the geodesic flow $\cG^t$ that $\eps$-shadow $(c_1, \vec c_1)$ and $(c_2, \vec c_2)$ are different. 
\end{theo}

The argument will be carried out on the cover $T^1\DD^2$ and we need to introduce some notation first.
The Poincar\'e disk $\DD^2$ covers the hyperbolic surface $(S,\mu)$ so that the covering map is a local isometry. Let $d_{\DD^2}$ denote the metric on ${\DD^2}$ induced by the hyperbolic Riemannian metric. We equip $T^1\DD^2$ with the product Riemannian metric (which is $\pi_1S$-equivariant) and denote by $d_{T^1\DD^2}$ the induced metric. 

\begin{proof} 
To prove the theorem, we lift all the structures to $T^1\DD^2$. We denote the
    lift of $f$ by $\tilde f\colon T^1\DD^2\to T^1\DD^2$. First note that it is
    sufficient to establish 
    the ``shadowing injectivity property" on $T^1\DD^2$.
    (Indeed, if a geodesic $(\gamma,\vec{\gamma})\colon\RR\to T^1\DD^2$ shadows a center leaf $(c,\vec{c})\colon \RR\to T^1\DD^2$, and its image under a Deck transformation $g (\gamma,\vec{\gamma})$ shadows another center leaf $(c',\vec{c'})$, then $(\gamma,\vec{\gamma})$ shadows both $(c,\vec{c})$ and $g^{-1} (c',\vec{c'})$. Hence, by shadowing injectivity on the cover we would have $(c',\vec{c'})=g(c, \vec{c})$ and, hence, posited injectivity on $T^1S$.)

    Note that according to our assumptions on subbundles of $f$ the size of
    local product structure neighborhoods is uniform in $\varepsilon$; that
    is, the size stays fixed when we let $\varepsilon\searrow 0$ (and
    $\delta\searrow 0$). Hence we assume that $\varepsilon$
    is much smaller than the size of local product structure neighborhoods.

Each center leaf of $\tilde f$ has the form $(c,\vec{c})\colon \RR\to T^1\DD^2$ where $c\colon\RR\to\DD^2$. Further, by dynamical coherence, each center leaf $(c, \vec{c})$ is contained in a unique center-unstable leaf $W^{cu}(c, \vec{c})$ which is uniformly transverse (\ie with angle bounded away from 0) to the circle fibers of $\pi$. It follows that $\pi\colon W^{cu}(c, \vec{c})\to \DD^2$ is a covering map and, hence, is a diffeomorphism. Further, because $\pi$ is a Riemannian submersion and the angle between $E^s\oplus E^{uu}$ and the circle fibers is uniformly bounded away from 0 (uniformly in $\varepsilon$) we also have that $\|D((\pi|_{W^{cu}(x)})^{-1})\|$ is uniformly bounded and, hence,
\begin{equation}
\label{eq_pi_bound}
d_{W^{cu}(x)}(x_1, x_2)\le C d_{\DD^2}(\pi(x_1),\pi(x_2)),
\end{equation}
where $x_2\in W^{cu}(x_1)$, $d_{W^{cu}(x)}$ is the metric induced by the restriction of the Riemannian metric to $W^{cu}(x)$ and $C$ is a constant which does not depend on $\varepsilon$.

We will argue by contradiction. Assume that a geodesic $(\gamma,\vec{\gamma})\colon\RR\to T^1\DD^2$ $\varepsilon$-shadows two distinct center leaves $(c_1, \vec{c_1})\colon\RR\to T^1\DD^2$ and $(c_2, \vec{c_2})\colon\RR\to T^1\DD^2$. Then, because $\pi$ does not increase distance and, using the triangle inequality, we have
\begin{equation}
\label{eq_2e}
d^\cH_{\DD^2}(c_1,c_2)\le d^\cH_{T^1\DD^2}((c_1, \vec{c_1}), (c_2, \vec{c_2}))\le 2\varepsilon,
\end{equation}
where superscript $\cH$ indicates the Hausdorff distance with respect to the corresponding metric.

Without loss of generality we can assume that $(c_1, \vec{c_1})$ and $(c_2, \vec{c_2})$ belong to the same center-unstable leaf. To see this consider the ``heteroclinic center leaf" $(c_3, \vec c_3)= W^{u, loc}(c_1, \vec c_1)\cap W^{s,loc}(c_2, \vec c_2)$. Clearly this leaf is also $K\varepsilon$-shadowed by $\gamma$, where $K$ is a constant associated with the local product structure and is independent of $\varepsilon$. Also $c_3$ is different either from $c_1$ or from $c_2$ (or both). We assume that $c_1$ and $c_3$ are distinct (otherwise the arguments are the same with the time direction reversed). Hence, we can conclude that, by choosing $\varepsilon$ smaller and by replacing $c_2$ with $c_3$, we indeed can assume that the center leaves $(c_1, \vec c_1)$ and $(c_2, \vec c_2)$ are contained in the same center-unstable leaf.

For any $n\ge 1$ consider center leaves $\tilde f^n(c_1, \vec c_1)$ and $\tilde f^n(c_2, \vec c_2)$. Let $\gamma_1$ and $\gamma_2$ be the geodesics which shadow them, respectively. By~(\ref{eq_2e}) and because $\|D\tilde f\|$ is uniformly bounded, we see that $\gamma_1$ and $\gamma_2$ are finite distance apart (as oriented geodesics).  Recall that in $\DD^2$ geodesics either coincide or diverge exponentially in positive or negative time. Hence $\gamma_1=\gamma_2$. By the same token as before~(\ref{eq_2e}) we also have
\begin{equation}
\label{eq_2e_n}
d^\cH_{\DD^2}( \pi\tilde f^n(c_1, \vec c_1), \pi\tilde f^n(c_2, \vec c_2))\le d^\cH_{T^1\DD^2}(\tilde f^n(c_1, \vec c_1), \tilde f^n(c_2, \vec c_2))\le 2\varepsilon
\end{equation}

Now for each center leaf consider its neighborhood in the center-unstable leaf
$$
W^{cu, loc}(c, \vec c)=\bigcup_{x\in(c,\vec c)}W^{u, loc}(x)\subset W^{cu}(c, \vec c),
$$
where the size of $W^{u, loc}$ is the size of local product structure neighborhoods and hence is much bigger than $\varepsilon$. For all $n\ge 0$, using~(\ref{eq_pi_bound}) and~(\ref{eq_2e_n}), we have
\begin{multline*}
d^\cH_{W^{cu}(\tilde f^n(c_1,\vec c_1))}(\tilde f^n(c_1, \vec c_1), \tilde f^n(c_2, \vec c_2))\le C d^\cH_{\DD^2}(\pi (\tilde f^n(c_1, \vec c_1), \pi \tilde f^n(c_2, \vec c_2))  \\
 \le C d^\cH_{T^1\DD^2}(\tilde f^n(c_1, \vec c_1), \tilde f^n(c_2,  \vec c_2))\le 2C\varepsilon
\end{multline*}
Recall that $C$ is independent of $\varepsilon$. Hence, by the above inequality, we can pick $\varepsilon$ sufficiently small so that $\tilde f^n(c_2, \vec c_2)$ belongs to $W^{u, loc}(\tilde f^n(c_1, \vec c_1))$ for all $n\ge 0$. Now pick an $x\in(c_1, \vec c_1)$ and let $y\in W^{u,loc}(x)\cap (c_2, \vec c_2)$. Then $\tilde f^n(x)$ and $\tilde f^n(y)$ diverge exponentially quickly along the unstable leaf. Pick the smallest $n\ge 1$ such that $\tilde f^n(y)\notin W^{u, loc}(\tilde f^n(x))$. On the other hand $\tilde f^n(y)\in \tilde f^n(c_2, \vec c_2)\subset W^{u,loc}(c_1, \vec c_1)$. Therefore there exists a point  $z\in (c_1, \vec c_1)$ such that $\tilde f^n(y)\in W^{u,loc}(\tilde f^n(z))$. We conclude that the unstable leaf $W^u(x)$ intersects the center leaf $(c_1, \vec c_1)$ in two distinct points $x$ and $z$. But such a configuration is impossible inside $W^{cu}(x)$ (which is diffeomorphic to $\DD^2$) by the standard Poincare-Bendixon's argument (see e.g. \cite[Section 14.1]{KH}). 
\end{proof}

The following global shadowing property of (general) Anosov flows is a straightforward consequence of \cite[Theorem 18.1.7]{KH}: 
\begin{theo} \label{t.globalshadowing} Let $X$ be an Anosov flow on a compact manifold $M$. Then for any $\varepsilon>0$ there is $\delta>0$ with the following property: 
consider  any $C^0$-foliation $\cF$ directed by a vector field $Y$ such that, for every $x\in M$, 
$$\|Y(x)-X(x)\|<\delta.$$
Then there is  $h\colon M\to M$ with the following property
\begin{itemize}
 \item  for any $x\in M$ one has $d(h(x),x)<\varepsilon$.
 \item for any leaf $L$ of $\cF$ the image $h(L)$ is exactly an orbit of $X$
 \item $h$ is continuous and onto: in particular, for any orbit $\gamma$ of $X$ there is a leaf $L$
 with $h(L)=\gamma$.
\end{itemize}
\end{theo}

Putting together Theorems \ref{t.globalshadowing} and~\ref{t.uniqueshadowing} we immediately obtain the following.

\begin{coro}\label{c.shadowingpartial}
For every $\eps>0$ there exists $\delta>0$ such that if $f: T^1S \to T^1S$ is a dynamically coherent partially hyperbolic diffeomorphism whose splitting is $\delta$-close to the one of $Z$, then, every center leaf $c:\mathbb{R} \to T^1S$ tangent to $X^c_f$ parametrized by arc-length is $\eps$-shadowed by a unique orbit of $Z$. Moreover, for every orbit $\gamma$ of $Z$ there is a unique center curve which is $\eps$-shadowed by $\gamma$. 
\end{coro}

\subsection{Dynamics at infinity}
 
We compactify the Poincar\'e disk $\DD^2$ by adding the ``the circle at infinity" $\partial \DD^2$ with the topology induced by the Euclidean topology of $\RR^2\supset \DD^2\cup\partial \DD^2$. Recall that $\partial\DD^2$ can be viewed as the space of asymptotic classes of geodesics in $\DD^2$. Further, each unit speed geodesic $\gamma$  in $\DD^2$ is  uniquely determined by an ordered pair of {\it boundary points} $(\gamma_-,\gamma_+)\in \partial \DD^2\times \partial \DD^2\backslash\Delta$, where
$$
\gamma_-=\lim_{t\to-\infty}\gamma(t),\,\,\,\,\, \gamma_+=\lim_{t\to+\infty}\gamma(t)
$$
In the same way any curve $c\colon\RR\to \DD^2$ which stays a bounded distance away from a geodesic determines a pair of boundary points $(c_-, c_+)\in \partial \DD^2\times \partial \DD^2\backslash\Delta$.

Now let $\Psi\colon S\to S$ be a homeomorphism and let $\tilde \Psi\colon\DD^2\to\DD^2$ be a lift. Then the action of $\tilde\Psi$ on geodesic rays induces a homeomorphism $\partial\tilde\Psi\colon \partial\DD^2\to \partial\DD^2$ which depends only on the isotopy class of $\Psi$ (but does depend on the choice of the lift). Namely, if $\Psi'\simeq\Psi$ and $\tilde\Psi'\simeq\tilde\Psi$ is the lift of the isotopy, then $\partial\tilde\Psi'=\partial\tilde\Psi$. Further if $\Psi$ is pseudo-Anosov, then for any lift $\tilde\Psi$ the boundary homeomorphism $\partial\tilde\Psi$ has finitely many fixed points on $\partial\DD^2$ alternatively expanding and contracting (see \eg~\cite[Theorem 5.5]{CB}). Moreover, the lift of the unstable geodesic lamination $\tilde L^u$ of the mapping class of $\Psi$ contains the geodesics joining consecutive expanding fixed points.

In the case when $\Psi$ is reducible, the structure of the fixed points at infinity can be more complicated, but the following proposition is still true.

\begin{prop} ( \cite[\textsection 5, \textsection 6]{CB},\cite[Section 9]{M})
\label{prop_geodesic}
Let $\Psi\colon S\to S$ be a pseudo-Anosov homeomorphism or a reducible homeomorphism with a pseudo-Anosov component. Then there is a positive iterate $\partial\tilde \Psi^n$ which has two distinct expanding  fixed points $\gamma_-$ and $\gamma_+$  and the geodesic $\gamma$ which is determined by these points belongs to the lift of the unstable geodesic lamination $\tilde L^u$. Further $\Psi^n$ can be isotoped to $\Phi$ such that corresponding lift $\tilde \Phi\simeq \tilde \Psi^n$ preserves the unstable geodesic lamination and hence preserves $\gamma$.
\end{prop}

\begin{rema}\label{rem_terminology}
We remark that there is a discrepancy between the terminology in hyperbolic dynamics and the terminology for laminations (and singular foliations) of pseudo-Anosov diffeomorphisms. Namely, the {\it stable geodesic lamination} is the one which {\it  expands } sufficiently large segments of its geodesics. We refer to books~\cite{CB, Cal} for further background on geodesic laminations and pseudo-Anosov theory.
\end{rema} 

\subsection{Dynamics at infinity and the center leaves}\label{ss.centerleafdyn}

Here we consider a dynamically coherent partially hyperbolic diffeomorphism $f\colon T^1S\to T^1S$ which satisfies the assumptions of Theorem~\ref{t.uniqueshadowing}. Assume that there is a diffeomorphism $\Phi\colon S\to S$ such that $f$ is isotopic to $D\Phi\colon T^1S\to T^1S$. Consider the $\pi_1S$-covers $\DD^2\to S$ , $T^1\DD^2\to T^1S$ and lifts $\tilde\Phi\colon \DD^2\to\DD^2$,  $\tilde f\colon T^1\DD^2\to T^1\DD^2$ such that the isotopy $f\simeq D\Phi$ lifts to an isotopy $\tilde f\simeq D\tilde\Phi$. As $\tilde f\simeq D\tilde\Phi$ is an equivariant isotopy, we have
\begin{equation}
\label{eq_bounded1}
d_{T^1\DD^2}(\tilde f, D\tilde\Phi)<C
\end{equation}
and because the bundle map $\pi\colon T^1\DD^2\to\DD^2$ is a Riemannian submersion, we also have
\begin{equation}
\label{eq_bounded2}
d_{\DD^2}(\pi\circ\tilde f,\tilde \Phi)<C.
\end{equation}
Consider the center leaves for $\tilde f$, \ie connected components of preimages of center leaves for $f$ under the covering map. Recall that each center leaf has the form $(c,\vec c)\colon \RR\to T^1\DD^2$ where $c\colon\RR\to\DD^2$. We slightly abuse terminology by referring to the underlying curve $c\colon \RR\to\DD^2$ as a center leaf as well. This is harmless because $c\colon \RR\to\DD^2$ uniquely determines the center leaf. By Theorem~\ref{t.uniqueshadowing} each center leaf $c$ is shadowed by a unique geodesic $\gamma$, and hence each center leaf $c$ is uniquely determined by a pair of points on the boundary and vice versa. 
According to the following claim the dynamics of $\tilde f$ on the space of center leaves is uniquely determined by the dynamics on the boundary.

\begin{claim}
\label{claim_center_dynamics}
For a center leaf $(c, \vec c)$ of $\tilde f$,
 let $(c_1,\vec c_1)=\tilde f(c,\vec c)$.
    Then
$$
(c_{1-},c_{1+})=(\partial\tilde\Phi(c_-), \partial\tilde\Phi(c_+)),
$$
where $(c_-,c_+)$ and $(c_{1-},c_{1+})$
are the boundary points of $c$ and $c_1$ respectively.
\end{claim}

\begin{proof}
Let $\gamma$ be the geodesic which shadows $c$. Then $\gamma$ and $c$ have the same boundary points $(c_-,c_+)$. Applying $\tilde f$, we have that $(c_1, \vec c_1)$ is bounded distance away from $\tilde f(\gamma, \gamma')$. By~(\ref{eq_bounded1}), $\tilde f(\gamma,\gamma')$ is bounded distance away from $D\tilde\Phi(\gamma,\gamma')=(\tilde\Phi(\gamma), \tilde\Phi(\gamma)')$. But $\tilde\Phi(\gamma)$ has boundary points $(\partial\tilde\Phi(c_-), \partial\tilde\Phi(c_+))$ by the definition of $\partial\tilde\Phi$.
\end{proof}

\subsection{Proof of Theorem~\ref{t.incoherent} and the Addendum} 

Let $f$ be a partially hyperbolic diffeomorphism and let $\varphi\in\cM(S)$ be the mapping class as in Theorem~\ref{t.incoherent} or the Addendum with $f\in\iota(\varphi)$. First we pass to positive iterates of $f$ and $\varphi$ so that Proposition~\ref{prop_geodesic} applies and yields a representative $\Phi$ of $\varphi$ such that the lift $\tilde \Phi$ preserves a geodesic $\gamma$. Further, $\gamma$ belongs to the unstable lamination and, hence, is {\it coarsely contracting} (cf. Remark~\ref{rem_terminology}), \ie for all sufficiently long segments $[\gamma(a),\gamma(b)]\subset \gamma(\RR)$ we have
$$
d_{\DD^2}(\tilde\Phi(\gamma(a)),\tilde \Phi(\gamma(b)))\le \lambda d_{\DD^2}(\gamma(a),\gamma(b)),
$$
where $\lambda<1$ and depends only on the mapping class, see \eg~\cite[Section 1.7.4]{Cal}.

Next assume that $f$ is dynamically coherent. By choosing sufficiently small $\eta>0$ we have that the rest of the assumptions of Theorem~\ref{t.uniqueshadowing} are also satisfied. By Corollary~\ref{c.shadowingpartial}, there exists a unique center leaf $(c,\vec c)\colon \RR\to T^1\DD^2$ which is shadowed by $\gamma$. Claim~\ref{claim_center_dynamics} implies that $(c,\vec c)$ is preserved by $\tilde f$. Using~(\ref{eq_bounded2}) and the fact that $\gamma$ is coarsely contracting under $\tilde\Phi$, it is easy to show that $(c,\vec c)$ is coarsely contracting under $\tilde f$. Hence, $\tilde f$ has a fixed point $x\in (c, \vec c)$.

Recall that the center foliation is obtained by taking the intersection of the center-stable and the center-unstable foliations. Consider the center-unstable leaf $L$ which contains $(c,\vec c)$. Because $E^c\oplus E^{ss}$ is transverse to the circle fibers of $T^1S$, the restriction of the bundle map $\pi|_L\colon L\to\DD^2$ is a diffeomorphism. Note that $\pi$ sends a center leaves to its underlying curve.  Pick a point $y\in L$ on the local stable manifold of $x$. Consider the sequence $\{\tilde f^n(y), n\ge 0\}$ and corresponding sequence of center leaves which contain these points $\{(c_n, \vec c_n); n\ge 0\}\subset L$. Because $\pi|_L$ is a diffeomorphism, $c_n\colon\RR\to\DD^2$ are all distinct and therefore have distinct pairs of boundary points $(c_{n-},c_{n+})$, $n\ge 0$. Each $c_n$ partitions $\DD^2$ into two connected components: one containing $\pi(x)$ and the other one containing $\pi(\tilde f^m(y))$ for $m<n$. It follows that one of the sequences $\{c_{n-}, n\ge 0\}$ and $\{c_{n+}, n\ge 0\}$ is non-increasing and the other one is non-decreasing (we orient $\partial\DD^2$ counter-clockwise). For the sake of concreteness, we can assume that $\{c_{n-}, n\ge 0\}$ is non-increasing and that $\{c_{n+}, n\ge 0\}$ is non-decreasing (see Figure~\ref{fig:2}). Let
$$
\ell_-=\lim_{n\to\infty} c_{n-},\,\,\,\,\, \ell_+=\lim_{n\to\infty} c_{n+}
$$ 

Also let $c_\infty$ be the center leaf determined by $(\ell_-,\ell_+)$.

\begin{figure}[ht]
\vspace{-0.5cm}
\begin{center}
\includegraphics[scale=0.9]{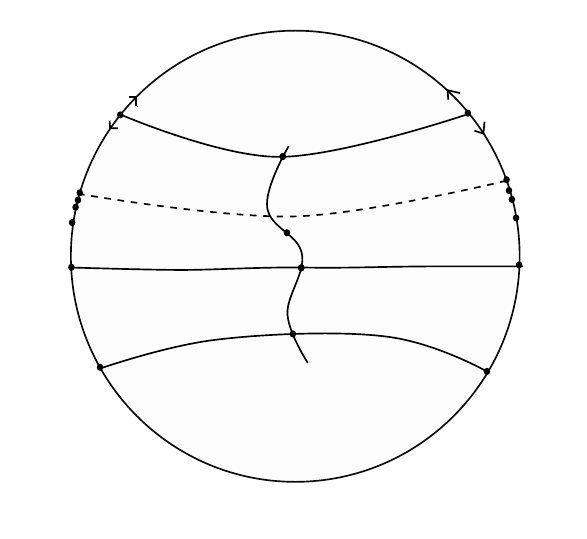}
\begin{picture}(0,0)
\put(-10,180){\small{$L$}}
\put(-48,183){\small{$c_+$}}
\put(-221,183){\small{$c_-$}}
\put(-184,179){\small{$(c,\vec c)$}}
\put(-131,169){\small{$x$}}

\put(-127,106){\small{$\tilde f(y)$}}
\put(-137,80){\small{$y$}}

\put(-34,153){\small{$\ell_+$}}
\put(-236,153){\small{$\ell_-$}}
\put(-100,152){\small{$(c_\infty,\vec c_\infty)$}}

\put(-30,113){\small{$c_{2,+}$}}
\put(-242,113){\small{$c_{2,-}$}}

\put(-41,70){\small{$c_{1,+}$}}
\put(-232,70){\small{$c_{1,-}$}}

\end{picture}
\end{center}
\vspace{-0.5cm}
\caption{\small{Depiction of the elements in the proof.}\label{fig:2}}
\end{figure}

The pairs $(c_{n-}, c_{n+})$ are all distinct and, hence at least one of these sequences does not stabilize at the limit point, say, $\{c_{n-}, n\ge 0\}$ does not stabilize. Recall that by Claim~\ref{claim_center_dynamics}, we have $c_{n-}=\partial\tilde\Phi^n(c_{0-})$. Hence $\ell_-$ is a contracting fixed point for $\partial\tilde\Phi$ (at least on one side). By Proposition~\ref{prop_geodesic}, $c_-$ and $c_+$ are expanding fixed points. Hence, $\ell_-\neq c_-$ and consequently $c_\infty$ and $c$ are disjoint. By construction, $c_\infty$ separates $c$ from the sequence $\{c_n; n\ge 0\}$. In particular, we have for all $n\ge0$
$$
d_{\DD^2}(\pi(x),\pi(\tilde f^ny))\ge d_{\DD^2}(\pi(x),c_\infty)=\rho>0,
$$
which yields a contradiction as $\tilde f^ny$ lies on the stable manifold of $x$.\hfill$\square$

\begin{rema}
 The proof can be extended to rule out a weaker version of dynamical coherence. Namely, we show that there does not exist an invariant foliation $\cF^{cu}$ tangent to $E^{c}\oplus E^{uu}$. Indeed, using a leaf of the center stable branching foliation (see~\cite{BI})\footnote{Notice that it is always possible to consider a ``lowermost" center-stable leaf which will be fixed by $\tilde f$.} and assuming that there is an invariant foliation tangent to $E^{c}\oplus E^{uu}$ one can use the same argument because the argument takes place inside a center-stable leaf. 
  \end{rema} 


\subsection{Further remarks and extensions}

Using the branching foliations of Burago-Ivanov~\cite{BI},
some of the intermediate results in the proof above may be established in 
the dynamically incoherent setting.
In particular, we obtain information about complete curves tangent to the center distribution as well as their coarse dynamics. 

Putting together Theorem \ref{t.globalshadowing}, the ideas in Theorem \ref{t.uniqueshadowing} and a key result from \cite{BI} we obtain the following.

\begin{prop}\label{p.shadowingpartial}
For every $\eps>0$ there exists $\delta>0$ such that if $f: T^1S \to T^1S$ is a partially hyperbolic diffeomorphism whose splitting is $\delta$-close to that of $\cG^t$, then every $c:\mathbb{R} \to T^1S$ tangent to $E^c_f$ and parametrized by arc-length is $\eps$-shadowed by a unique orbit of $\cG^t$. Moreover, for every orbit $\gamma$ of $\cG^t$ there is a unique center curve which is $\eps$-shadowed by $\gamma$. 
\end{prop}

\begin{proof} If the center bundle $E^c_f$ integrates to an invariant foliation then the result is a direct consequence of Theorems~\ref{t.globalshadowing} and~\ref{t.uniqueshadowing} (which in particular give that the center bundle should be uniquely integrable).

Now consider the case when $E^c_f$ does not integrate to a foliation. Then the existence of a unique orbit of $\cG^t$ shadowing a given center curve is given by Theorem \ref{t.globalshadowing} and the fact that no two center curves are shadowed by the same orbit of $\cG^t$ is given by Theorem \ref{t.uniqueshadowing}. Indeed, in Theorem \ref{t.uniqueshadowing} one does not actually need that the diffeomorphism $f$ is dynamically coherent: the one important point where dynamical coherence was used is to reduce to the case when the two center curves which were shadowed by the same orbit of the flow were in the same center unstable leaf (or center stable leaf). This can be done in absence of dynamical coherence by appealing to the fundamental results of \cite{BI}. In particular,~\cite[Proposition 3.1]{BI} implies that saturating by strong stable (resp. unstable) manifolds gives a surface tangent to the center-stable bundle (resp. center-unstable). Then the main technical result of~\cite{BI} provides branching foliations which can be used to replace the center-stable and center unstable leaves.

To show that every orbit $\gamma$ of $\cG^t$ shadows some center curve, we apply \cite[Key Lemma]{BI} to approximate $E^s_f \oplus E^c_f$ and $E^c_f \oplus E^u_f$ by integrable distributions $E_n$, $F_n$. Moreover, \cite[Theorem 7.2]{BI} implies that the integral leaves of $E_n$ are pairwise at bounded distance in $T^1\mathbb{D}$ (and similarly for $F_n$). The intersection gives a one-dimensional foliation $\mathcal{F}^c_n$ and as $n\to \infty$ the leaves converge uniformly to center curves. By applying Theorem~\ref{t.globalshadowing}, we have that for all sufficiently large $n$ there exists a unique leaf $\ell_n$ in $\mathcal{F}^c_n$ which is shadowed by $\gamma$. As $n\to \infty$ the curves $\ell_n$ converge to the posited center leaf.  
\end{proof}

\begin{figure}[ht]
\vspace{-0.5cm}
\begin{center}
\includegraphics[scale=0.5]{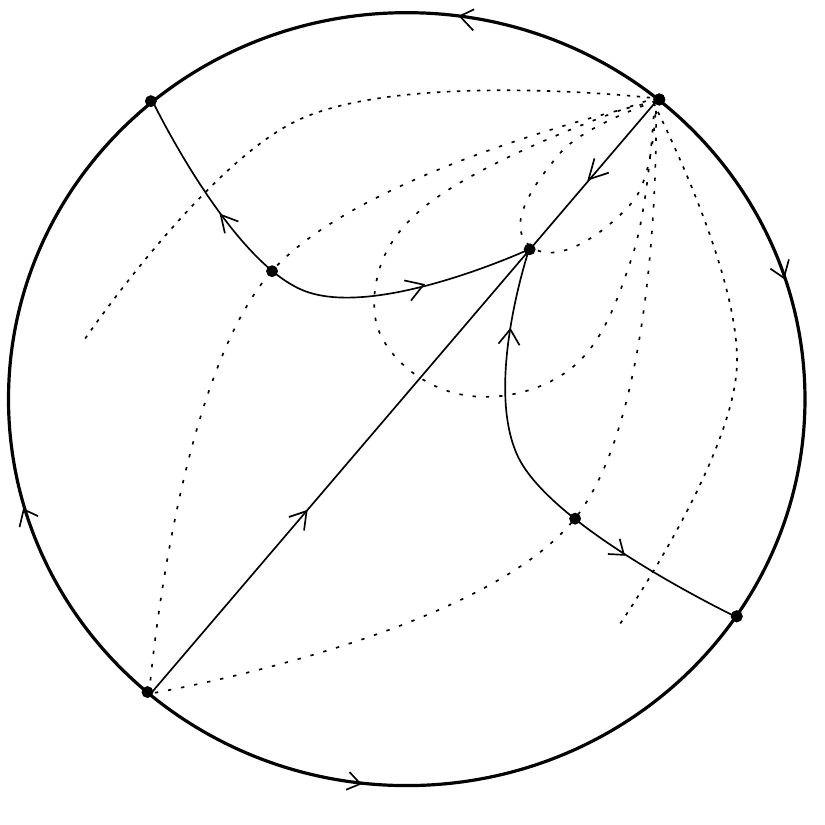}
\begin{picture}(0,0)

\end{picture}
\end{center}
\vspace{-0.5cm}
\caption{\small{Plausible behavior on a periodic $cs$-leaf associated to a regular periodic point of the pseudo-Anosov. The dotted lines correspond to the strong stable manifolds.}\label{fig:3}}
\end{figure}

Using this improvement, one can obtain the following.

\begin{coro}\label{c.noncomplete} Let $f$ be a diffeomorphism which satisfies the assumptions of Theorem~\ref{t.incoherent}. Then there exists a center leaf $c$ such that the union of the stable manifolds through $c$ does not form a complete surface.  \end{coro}

\begin{proof}
Let $c$ be the center leaf constructed in the proof of Theorem~\ref{t.incoherent}, \ie the center leaf which is shadowed by the geodesic which corresponds to two repelling fixed point on the boundary. Let $W^s(c)=\bigcup_{y\in c} W^{ss}(y)$. According to~\cite{BI} one knows that  $W^s(c)$ is a surface tangent to $E^{ss}_f \oplus E^c_f$. Hence its lift $\tilde W^s(c)$ to $T^1\mathbb{D}^2$ is transverse to the circle fibers. By a similar argument as in Theorem~\ref{t.incoherent}, we see that $\tilde W^s(c)$ cannot be a covering of $\mathbb{D}^2$. Indeed, one can choose a center leaf in $\tilde W^s(c)$ whose endpoints leave the attracting fixed point at infinity in between; the fact that this center curve belongs to $\tilde W^s(c)$ gives a contradiction with the dynamics induced in the boundary. In particular, it cannot be complete as a surface in $T^1S$.  
\end{proof}

See Figure~\ref{fig:3} for a possible behavior inside certain periodic center-stable leaf. Finally, we make some comments on the dynamics of the center curves related to~\cite[Problem 7.26]{BDV}.

\begin{rema}\label{rem.centerleafdyn}
Notice that the considerations about coarse dynamics of center leaves made in Section~\ref{ss.centerleafdyn} all apply to the center curves even when the diffeomorphism is dynamically incoherent. In particular, using the dynamics at infinity of the mapping class one can simultaneously create center curves which are contracting and expanding at infinity (non-compact center leaves of stable and unstable type in the nomenclature of~\cite[Problem 7.26]{BDV}). Further,  one can also show the following facts:
\begin{itemize}
\item If the mapping class $\Phi$ is pseudo-Anosov, then there are no closed periodic center curves. 
\item When two center curves merge as shown in Figure~\ref{fig:3}, one can see that there is an invariant center curve with coarse saddle-node behaviour and fixed points. 
\end{itemize} 
\end{rema}

\subsection{Minimality of the strong foliations}\label{ss.minimal}
In~\cite{BDU} it was shown that amongst robustly partially hyperbolic diffeomorphisms on $3$-manifolds, those for which one of the two strong foliations is minimal form an open and dense subset. The proof of~\cite{BDU} is done by performing certain $C^1$ perturbations and robust transitivity is used to avoid checking that these perturbations are still transitive. Because all the perturbations made in~\cite{BDU} (see also \cite[Section 7.3.3]{BDV}) can be made conservative\footnote{The perturbations needed are: the closing lemma, $C^1$-connecting lemma (\cite{BC}) and creation of blenders (\cite{BD,HHTU}).}, the same proof immediately applies to yield the following statement. There exists an open set $\cU$ of conservative partially hyperbolic diffeomorphisms which are stably ergodic, satisfy the conclusions of Theorem~\ref{t.incoherent} and for which either the strong stable foliation or the strong unstable foliation is minimal.

In our setting we can use the specific knowledge on the dynamics on center leaves (see Remark \ref{rem.centerleafdyn}) together with the arguments of \cite{BDU} to show the following stronger result. 

\begin{prop}\label{p.minimal} 
Let $f$ be a diffeomorphism which satisfies the assumptions of Theorem~\ref{t.incoherent}. Then, there exists a small conservative perturbation $\hat f$ of $f$ and a $C^1$-neighborhood $\cU$ of $\hat f$ such that for every $g \in \cU$ the strong stable and unstable foliations are both minimal. In particular, $\hat f$ is stably ergodic and robustly transitive. 
\end{prop}

\begin{proof} 
Without loss of generality we can assume that $f$ has dense periodic points all of which are hyperbolic (see e.g. \cite{BC}). 

From stable ergodicity of $f$ we know that there is a neighborhood $\cU_0$ of $f$ such that for every volume preserving $g \in \cU_0$ and every periodic point $p$ of $g$ of stable index $2$ the stable manifold $W^s(p)$ is dense in $T^1S$ and symmetrically for the unstable manifolds of a periodic point $q$ of $g$ of stable index $1$ (see \cite[Proposition 2.1]{BDU}, notice that the proof there does not use any perturbation, just that $f$ is transitive and partially hyperbolic). 

Take $p,q$ to be periodic points of $f$ which belongs to a saddle node center leaf (see \S \ref{ss.centerleafdyn}) and are respectively the first and the last periodic points in that leaf. For the sake of concreteness assume that $p$ has stable index $2$ and $q$ has stable index $1$. Using the fact that the strong stable and strong unstable foliations are close to those of the geodesic flow which are minimal, we can assume that every strong stable leaf of some length $L>0$ intersects the unstable manifold of $q$ while every strong unstable manifold of length $L$ intersects the stable manifold of $p$. In the terminology of \cite{BDU}, this means that the stable manifold of $p$ is a $u$-section and the unstable manifold of $q$ is a $s$-section (see \cite[Section 3]{BDU}).  

One can create blenders associated to $p$ and $q$ by conservative perturbation of $f$ (see \cite{HHTU}). This implies that in a neighborhood $\cU_1 \subset \cU_0$ of $f$ we have that for every $g\in \cU_1$ we know that the strong unstable manifold of $p_g$ activates the blender of $q_g$ and vice versa. In particular, we obtain for every $g \in \cU_1$ which preserves volume that the strong stable manifold of $q_g$ and the strong unstable of $p_g$ are dense in $T^1S$ (see~\cite[Proposition 2.6]{BDU}). The fact that every strong stable manifold of length $L$ intersects the unstable manifold of $p_g$ implies minimality of the strong unstable foliation for each conservative $g\in \cU_1$ and the same argument gives minimality of the strong stable foliation (for more details, see \cite[Proposition 4.1]{BDU}). 

It follows from minimality of the strong unstable foliation that there exists $L>0$ such that every strong unstable manifold intersects a blender. This implies that each conservative $g \in \cU_1$ is in the hypothesis of the main theorem of \cite{PS}. As the other strong stable foliation is minimal, this implies that it is robustly minimal even among perturbations which may not be conservative. The same argument applies to the strong unstable foliation and completes the proof. 
\end{proof}

\begin{rema}
Notice that this implies that the set of branching points for the center curves is dense in $T^1S$. 
\end{rema}

\begin{rema} 
This argument applies to the other examples we have constructed in Theorem~\ref{t.geodesic} thanks to Remark~\ref{rema-circle} (see also the discussions in \cite{BDU} and \cite[Section 7.3.3]{BDV}). 
\end{rema}

{\small \emph{Acknowledgments:} The fourth author would like to thank Thomas Barthelem\'e, Sergio Fenley and Steven Frankel for very helpful discussions.}

\end{document}